\documentclass[a4paper, 10pt]{article}

\usepackage[english]{babel}
\usepackage{amssymb, amsmath, amstext, amsopn, amsthm, amscd, amsxtra, amsfonts}
\usepackage[all]{xy}

\usepackage{multicol}
\usepackage[T1]{fontenc}
\usepackage[utf8]{inputenc}
\usepackage{pstricks}
\usepackage{colonequals}

\usepackage[bookmarks,bookmarksnumbered,plainpages=false]{hyperref}

\swapnumbers
\newtheoremstyle{standard}
 {16pt}  %Abstand nach oben
 {16pt}  %Abstand nach oben
 {}  %Schrift Text
 {}  %Einr??cken
 {\bfseries}%\normalfont} %Schrift Kopf
 {}  %Zeichen nach Kopf
 { } %Abstand
 {{\thmname{#1~}}{\thmnumber{#2.}}\thmnote{~(#3)}} %Format

\newtheoremstyle{kursiv}
 {16pt}  %Abstand nach oben
 {16pt}  %Abstand nach oben
 {\itshape}  %Schrift Text
 {}  %Einr??cken
 {\bfseries}%\normalfont} %Schrift Kopf
 {}  %Zeichen nach Kopf
 { } %Abstand
 {{\thmname{#1~}}{\thmnumber{#2.}}\thmnote{~(#3)}} %Format

\theoremstyle{standard}
\newtheorem{defn} [subsection]{Definition}

\newtheorem{rem}   [subsection]{Remark}

\newtheorem{numba} [subsection]{}

\theoremstyle{definition}

\theoremstyle{kursiv}
\newtheorem{thm}[subsection]{Theorem}
\newtheorem{prop} [subsection]{Proposition}
\newtheorem{cor} [subsection]{Corollary}
\newtheorem{lem} [subsection]{Lemma}

\newcommand{\Punkt}{\nopagebreak\hspace*{\fill}$\Box$}

\newcommand{\im}{\mathrm{im}}
\newcommand{\Ad}{\mathrm{Ad}}

\newcommand{\evol}{\mathrm{evol}}

\newcommand{\id}{\mathrm{id}}

\newcommand{\N}{\mathbb{N}}

\newcommand{\R}{\mathbb{R}}
\newcommand{\K}{\mathbb{K}}
\newcommand{\C}{\mathbb{C}}

\newcommand{\fF}{\ensuremath{\mathcal{F}}}
\newcommand{\cC}{\ensuremath{\mathcal{C}}}

\newcommand{\set}[1]{\{  #1 \}}
\newcommand{\setm}[2]{\left\{\, #1 \middle\vert #2\,\right\}}
\newcommand{\norm}[1]{\lVert #1 \rVert}

\newcommand{\coloneq}{\colonequals}

\DeclareMathOperator{\Germ}{Germ}
\DeclareMathOperator{\bHol}{BHol}
\DeclareMathOperator{\op}{op}

\title{Complexifications of infinite-dimensional manifolds and new
constructions of
infinite-dimensional Lie groups}
\author{Rafael Dahmen, Helge Gl{\"o}ckner, Alexander Schmeding}
\begin{document}
%
%

% Kopf
\thispagestyle{empty}
\maketitle

$\;$\\[-15mm]
\begin{abstract}
Let $M$ be a real analytic manifold modeled on a locally convex space
and $K\subseteq M$ be a non-empty compact set.
We show that if an open neighborhood~$\Omega$ of~$K$ in~$M$
admits a complexification~$\Omega^*$ which is a regular topological space,
then the germ of~$\Omega^*$ around~$K$ (as a complex manifold)
is uniquely determined.
If $M$ is regular and the complexified modeling space of~$M$ is normal,
then a regular complexification $\Omega^*$ exists
for some~$\Omega$.
For each regular $\K$-analytic manifold~$M$
modeled on a metrizable locally convex space
over $\K\in \{\R,\C\}$
and each $\K$-analytic Banach-Lie group~$H$,
this enables
%us to turn
the group $\text{Germ}(K,H)$
of germs of $\K$-analytic $H$-valued maps around
$K$ in~$M$ to be turned into a $\K$-analytic Lie group
which is regular in Milnor's sense (and, actually, $C^0$-regular).
Notably, this provides a
$C^0$-regular real analytic Lie group structure
on the group $C^\omega(M,H)$
of $H$-valued real analytic maps on a compact real analytic manifold~$M$
(which, previously,
had only been treated in the convenient setting of analysis).
Combining our results concerning Lie groups of germs
with an idea by Neeb and Wagemann,
%we also obtain a
it is also possible to obtain a $C^0$-regular Lie group structure
on $C^\omega(\R,H)$.\vspace{1mm}
\end{abstract}

% Local Lie groups 
22E67, % Loop groups and related constructions, group-theoretic treatment
26E05, % Real-analytic functions
26E15, % Calculus of functions on infinite-dimensional spaces
46E40, % Spaces of vector- and operator-valued functions
46E50, % Spaces of differentiable or holomorphic functions on infinite-dimensional spaces
46G20, % Infinite-dimensional holomorphy
46T25, % Holomorphic maps in nonlin FA
58C10\\[2.5mm]% Holomorphic maps 
\vspace{-1mm}

\tableofcontents
\section*{Introduction and statement of results}
It is well known that every paracompact, finite-dimensional
real analytic manifold~$M$ admits a complexification,
whose germ around~$M$ (as a complex manifold) is uniquely determined.
Two classical constructions for such complexifications are known.
The first one, by Bruhat and Whitney \cite{BW1959},
is based on a glueing construction:
complex analytic extensions of real analytic charts are pasted together
in a suitable way.
The second approach was devised by Grauert \cite{grauert}.
Here the manifold is embedded in
its cotangent bundle which supports a complex analytic structure
(of a Stein manifold) on a neighborhood of the embedding. These neighborhoods
are the famous ``Grauert tubes,''
and one advantage is that these can be chosen contractible to~$M$.\\[2mm]
In this paper, we investigate complexifications around a non-empty compact
subset of an infinite-dimensional real analytic manifold.
One cannot expect to adapt Grauert's construction
to infinite-dimensional manifolds,
and, in fact,
one would not always expect to find (global) complexifications,
in the natural sense:\footnote{Later, we shall find it convenient to allow
also complexifications which are only diffeomorphic to
those just described, see Definition~\ref{cxgeneral}.}\\[2.5mm]
{\bf Definition.}
Let $M$ be a real analytic manifold modeled on a locally convex space~$E$.
A complex analytic manifold~$M^*$ modeled on~$E_{\mathbb C}$ is called
a \emph{complexification of~$M$} if $M\subseteq M^*$
and each $x\in M$ is contained in the domain~$U^*$ of a chart
$\psi\colon U^*\to U'\subseteq E_{\mathbb C}$
of~$M^*$ such that $\psi(M\cap U^*)=E\cap U'$ and
$\psi|_{M\cap U^*}\colon M\cap U^* \to E\cap U'$
is a chart for~$M$.\\[2.5mm]
Nonetheless,
we shall see that a complexification can be constructed at least
for \emph{an open neighborhood} of
a compact set, if the manifold and its model space satisfy certain
regularity properties. Our argument
adapts the line of thought developed by Bruhat and Whitney in
the classical paper \cite{BW1959}. In particular,
we obtain the following theorem: 
\paragraph{Theorem A} \emph{Let $M$ be a real analytic manifold modeled on the locally
convex space~$E$. Assume that as topological spaces,~$M$ is regular and~$E_\C$
is normal.
For each non-empty compact subset~$K$ of~$M$, there is an open
neighborhood~$\Omega$ of~$K$ in~$M$ such that~$\Omega$ admits
a complexification~$\Omega^*$.
The complex analytic manifold $\Omega^*$ constructed is a regular topological space.}\\[1em]
We also obtain a uniqueness result for complexifications:
\paragraph{Theorem B} \emph{Let $M$ be a real analytic manifold modeled on a locally convex space.
\begin{itemize}
\item[{\normalfont (a)}]
If $M^*_j$ is a complexification
of~$M$ for $j\in\{1,2\}$
such that~$M$ is a closed subset of~$M^*_j$
and~$M_j^*$ is paracompact,
then there exist open neighborhoods~$U_j$ of~$M$ in~$M_j^*$
and a complex analytic diffeomorphism $\psi\colon U_1\to U_2$
such that $\psi|_M=\id_M$.
\item[{\normalfont (b)}]
If $K\subseteq M$ is a non-empty compact set
and $\Omega^*_j$, for $j\in\{1,2\}$,
is a regular complexification
of an open neighborhood~$\Omega_j$ of~$K$ in~$M$
such that $\Omega_j\subseteq \Omega_j^*$,
then there exist open neighborhoods $U_j\subseteq\Omega_j^*$ of~$K$
and a complex analytic diffeomorphism $\psi\colon U_1\to U_2$
such that $\psi|_{U_1\cap \Omega_1}=\id_{U_1\cap \Omega_1}$.
\end{itemize}}
\noindent
We mention that closedness of~$M$ in a complexification~$M^*$
is often easily achieved. Indeed,
if~$M^*$ is a complexification of~$M$
and each open neighborhood of~$M$ in~$M^*$
contains a paracompact open neighborhood of~$M$,
then there exists a paracompact open neighborhood~$U$ of~$M$ in~$M^*$
such that~$M$ is closed in~$U$ (Proposition~\ref{specialize}\,(a)).
The paracompactness is automatic if~$M^*$ is metrizable
(see \cite[Theorem~5.1.3]{Engelking1989}).
If~$M$ is closed in~$M^*$
and $M^*$ is paracompact,
then there exists an anti-holomorphic involution $\sigma\colon U\to U$
on an open neighborhood~$U$ of~$M$ in~$M^*$ such that
\[
M=\{z\in U\colon \sigma(z)=z\}.
\]
A similar result is available for complexifications around compact sets
(see Proposition~\ref{specialize} (b) and (c) for details).\\[2.3mm]
Our applications in infinite-dimensional Lie theory concern Lie groups
of analytic Lie group-valued mappings (or germs
of such), and their regularity properties.
A Lie group~$G$ modeled on a locally convex space
is called \emph{$C^0$-regular}
if each continuous curve $\gamma\colon [0,1]\to L(G)$
in its Lie algebra $L(G)=T_1(G)$ admits a continuously
differentiable left evolution $\eta=\eta_\gamma\colon [0,1]\to G$
determined by
\[
\eta(0)=1\quad\mbox{and}\quad (\forall t\in [0,1])\;\; \eta'(t)=\eta(t).\gamma(t),
\]
and the evolution map $\evol_G\colon C^0([0,1],L(G))\to G$, $\evol(\gamma):=\eta_\gamma(1)$ is smooth (i.e., $C^\infty$).
Here, the dot denotes the natural
left action $G\times TG\to TG$ of~$G$ on its tangent bundle
(see \cite{hg2013}, \cite{hg2012}, \cite{NaS}; cf.\ \cite{Milnor1984}
and \cite{Neeb2006} for the related weaker notion
of ``regularity,'' and some of its applications).\\[2.5mm]
Let $H$ be a Banach-Lie group over
$\K\in \{\R,\C\}$, $L(H)$ be its Lie algebra,
$M$ a $\K$-analytic manifold modeled on a metrizable
locally convex space~$X$,
and $K\subseteq M$ a non-empty, compact
subset.
We write $\Germ_\K(K,H)$ for
the group of all germs $[\gamma]$ around~$K$
of $\K$-analytic maps $\gamma\colon U\to H$,
defined on an open neighborhood $U\subseteq M$ of~$K$.\\[2.4mm]
If $\K=\C$,
then $\Germ_\C(K,L(H))$
carries a natural locally convex
vector topology (see Proposition~\ref{prop_GermKZ})
making it an (LB)-space, namely
the locally convex direct limit of the Banach spaces
$\text{BHol}(U_n,L(H))$ of complex analytic, bounded
$L(H)$-valued functions on a basis $U_1\supseteq U_2\supseteq\cdots$
of neighborhoods of~$K$ in~$M$.
% (endowed with the supremum norm).
We show:\vspace{-1mm}
\paragraph{Theorem C} \emph{$\Germ_\C(K,H)$
can be made a $C^0$-regular complex analytic Lie group
with Lie algebra $\Germ_\C(K,L(H))$.}\\[2.5mm]
If $\K=\R$ and~$M$ is regular as a topological space,
then a complexification
$M^*$ of~$M$ can be used to make $\Germ_\C(K,L(H)_\C)$
a complex locally convex space.
We turn $\Germ_\R(K,L(H))$
into a real locally convex space
such that $\Germ_\C(K,L(H)_\C)=\Germ_\R(K,L(H))_\C$.
Then the following results can be obtained
(proofs of which will be provided in a later version
of this preprint):
%We show:\vspace{-1mm}
%
\paragraph{Theorem D} \emph{Let $M$ be a real analytic
manifold modeled
on a metrizable locally convex space,
$K\subseteq M$ be a non-empty compact set
and~$H$ be a Banach-Lie group over $\K\in\{\R,\C\}$.
If $M$ is regular as a topological space,
then $G:=\Germ_\R(K,H)$
can be made a $C^0$-regular $\K$-analytic Lie group
with Lie algebra $\Germ_\R(K,L(H))$,
such that the evolution map
$\evol\colon C^0([0,1],L(G))\to G$ is $\K$-analytic.}\\[2.5mm]
Taking $K=M$,
%we obtain as a
the following
special case is obtained:\vspace{-1mm}
\paragraph{Theorem E} \emph{Let $M$ be a compact real analytic
manifold and $H$ be a Banach-Lie group over $\K\in \{\R,\C\}$.
Then the group $G:=C^\omega(M,H)$ of all real analytic $H$-valued
maps on~$M$ can be made a $C^0$-regular
$\K$-analytic Lie group with Lie algebra
$C^\omega(M,L(H))$,
such that the evolution map
$\evol\colon C^0([0,1],L(G))\to G$ is $\K$-analytic.}\\[2.5mm]
In the special case that $M=X$
is a metrizable locally convex space,
the Lie group structure on $\Germ_\K(K,H)$
was already provided in~\cite{hg2004}.
The $C^0$-regularity was established
in \cite[Theorem~5.1.1]{dahmen2011}, if $M=X$ is a Banach space
(using tools from \cite{dahmen2010}).
For~$M$ a finite-dimensional $\K$-analytic manifold,
the Lie group structure on~$\Germ_\K(K,H)$
was constructed in~\cite{kaur2013}.
For real analytic compact~$M$,
a Lie group structure on $C^\omega(M,H)$
was provided in~\cite{KaM},
which is real analytic in the (weaker) sense
of convenient differential calculus.\footnote{Recall from
\cite[pp.\ 104--105]{KaM} that such mappings need
not be real analytic, as they need not
admit Taylor expansions. The counter-example
is defined on an (LB)-space.}\\[2.5mm]
Recently, Neeb and Wagemann~\cite{NaW}
obtained a regular infinite-dimensional
Lie group structure on $C^\infty(\R,H)$
for each regular Lie group~$H$,
and more generally on $C^\infty(\R^n\times K,H)$
if $n\in \N$ and~$K$ is a compact smooth manifold.\footnote{Compare
\cite{pieper} for
a detailed and streamlined
discussion of the special case $C^\infty(\R,H)$ (as well as weighted analogs),
and \cite{zaareer2013} for a
treatment of $C^k(\R,H)$.}
Since $C^\infty(\R,H)=C^\infty(\R,H)_*\rtimes H$
with $C^\infty(\R,H)_*:=\{\gamma\in C^\infty(\R,H)\colon \gamma(0)=1\}$,
the main point is to make $C^\infty(\R,H)_*$
a Lie group. To achieve this,
Neeb and Wagemann noted that the left logarithmic derivative
$(\delta^\ell\gamma)(t):=\gamma(t)^{-1}.\gamma'(t)$ provides a bijection
\[
\delta^\ell\colon C^\infty(\R,H)_*\to C^\infty(\R,L(H)),\quad \gamma\mapsto \delta^\ell\gamma
\]
which can be used as a global chart for a smooth manifold
structure on $C^\infty(\R,H)_*$ that turns the group
operations into smooth maps.
Combining a variant of this idea
with our results concerning spaces of germs,
one can show:
%we obtain:
%
\paragraph{Theorem F} \emph{If $H$ is a real Banach-Lie group,
then the group $C^\omega(\R,H)$ of all real analytic maps
$\gamma\colon \R\to H$
is a $C^0$-regular
smooth Lie group with Lie algebra
$C^\omega(\R,L(H))$.
If $\K=\C$ and $H$ is a complex
Banach-Lie group
or $\K=\R$ and $H$ is a real Banach-Lie group
such that $L(H)_\C=L(H_\C)$ for some
complex Banach-Lie group $H_\C$,
then $G:=C^\omega(\R,H)$ is a $C^0$-regular $\K$-analytic
Lie group and the evulution map
$\evol\colon C^0([0,1],L(G))\to G$ is $\K$-analytic.}\\[3mm]
The article is structured as follows:
After a preliminary Section~\ref{sec-prel},
we discuss existence and uniqueness
of complexifications in Sections~\ref{sec-unique}
and \ref{sec-exists}.
Locally convex spaces of complex analytic germs and their regularity (as a locally convex direct limit) are discussed in Section~\ref{sec-germ-compactly regular}
and then used in Section~\ref{sec-reg-cx} to construct the $C^0$-regular Lie group of complex analytic germs.
For later use, also the case of \emph{local} Lie groups~$H$ (see Definition \ref{defn_local_lie_group}) needs to be addressed there.
%On this foundation,
%Lie groups of real analytic functions and germs
%are discussed in Section~\ref{sec-real},
%and $C^\omega(\R,H)$
%in Section~\ref{sec-line}.
%
%
%
%
%
%
\section{Preliminaries and notation}\label{sec-prel}
In this section, we recall preliminaries
and fix some notation, for later use.\\[2.3mm]
By a \emph{locally convex space},
we mean a locally convex Hausdorff topological
vector space.
As usual, if~$X$ is a topological space and $Y\subseteq X$,
then a subset $U\subseteq X$ is called an \emph{open neighborhood}
of~$Y$ in~$X$ if~$U$ is open in~$X$ and $Y\subseteq U$.
\begin{defn}
 Let $E$ be a real locally convex space. Endow the real locally convex
 space $E_\C \coloneq E \times E$ with the operation 
    \begin{displaymath}
     (x+iy).(u,v) \coloneq (xu-yv, xv+yu) \quad \text{ for }
 x,y \in \R, u,v \in E.
    \end{displaymath}
 The complex locally convex space $E_\C$ obtained in this way is called the
 \emph{complexification} of $E$. We identify $E$ with the subspace
 $E \times \set{0} \subseteq E_\C$.
\end{defn}

\begin{defn}[{Complex analytic maps \cite{BS71b}}]
 Let $E$ and $F$ be complex locally convex spaces.
 A map $f \colon U \rightarrow F$ on an open subset $U \subseteq E$ is
 \emph{complex analytic} if it is continuous and admits locally a power series
 expansion around each $a \in U$, i.e.\ there exist continuous homogeneous
 polynomials
 $p_k\colon E\to F$ of degree $k$, such that 
   \begin{displaymath}
    f(x) = \sum_{k=0}^\infty p_k (x-a)
   \end{displaymath}
 pointwise for all $x$ in a neighborhood of $a$ in $U$.
\end{defn}
\noindent
See \cite{hg2002} and \cite{GaN}
for the following definition (cf.\ also \cite{Milnor1984}).
\begin{defn}[{Real analytic maps}]
Let $E$ and $F$ be locally convex spaces over~$\R$.
A map $f \colon U \rightarrow F$ on an open subset $U \subseteq E$
is called \emph{real analytic} or $C_\R^\omega$  if it extends to a complex
analytic map $g \colon W \rightarrow F_\C$ on some open neighborhood
$W \subseteq E_\C$ of $U$. 
\end{defn}
\noindent
Throughout the article, the word ``analytic manifold''
means an analytic manifold modeled on a locally convex space
(as discussed, e.g., in \cite{hg2002} of \cite{GaN}).
\begin{defn}
If~$F$ is a complex locally convex space,
we write $F_{\op}$ for~$F$, endowed with the opposite
complex structure
(thus multiplication with~$i$ on~$F_{\op}$ is given by multiplication
with $-i$ in~$F$).
If $M$ is a complex analytic manifold modeled on~$F$,
with atlas ${\mathcal A}$ of charts $\phi\colon U_\phi\to V_\phi\subseteq F$,
we write $M_{\op}$ for~$M$,
endowed with the complex analytic manifold
structure given by atlas of charts
$\phi\colon U_\phi\to V_\phi\subseteq F_{\op}$
for $\phi\in{\mathcal A}$.
A mapping $f\colon M\to N$
between complex analytic manifolds is called
\emph{anti-holomorphic} if it is complex analytic as
a map $M\to N_{\op}$ (or, equivalently,
as a map $M_{\op}\to N$).
\end{defn}
\begin{rem}
Let $M$ be a complex analytic manifold
modeled on a complex locally convex space~$F$
such that $F=E_\C$ for a real
locally convex space~$E$.
Let $\tau\colon E_\C\to E_\C$ denote the
complex conjugation given by $\tau(x+iy):=x-iy$
for all $x,y\in E$.
Then $\tau\colon (E_\C)_{\op}\to E_\C$
is an isomorphism of complex
locally convex spaces.
The maps
$\tau\circ \phi\colon U_\phi\to \tau(V_\phi)\subseteq F$
form an atlas for $M_{\op}$ (modeled on~$F$),
if we let $\phi\colon U_\phi\to V_\phi\subseteq F$
range through an atlas for~$M$.
\end{rem}
\noindent
This readily entails:
\begin{rem}\label{complex-op}
If~$M$ is a real analytic manifold modelled on~$E$
and $M^*$ a complexification of~$M$,
then also $(M^*)_{\op}$ is a complexification of~$M$.
Indeed, if $x\in M$ and $\phi\colon U^*\to U'$
is a chart of~$M^*$ around~$x$ with $\phi(M\cap U^*)=E\cap U'$,
then $\tau\circ \phi\colon U^*\to \tau(U')$
is a chart of $(M^*)_{\op}$ around~$x$ which takes
$M\cap U^*$ onto $E\cap \tau(U')$.
\end{rem}
\noindent
We recall a version of the well-known Identity Theorem
for analytic functions (cf.\ \cite{GaN}).
\begin{lem}[Identity Theorem]\label{idthm}
\begin{itemize}
\item[{\rm(a)}]
Let $M$ and $N$ be $\K$-analytic manifolds
modeled on locally convex spaces
$($where $\K\in \{\R,\C\})$
and $f_j \colon M\to N$ be $\K$-analytic maps
for $j\in \{1,2\}$.
If $M$ is connected and $f_1|_U=f_2|_U$
for some non-empty open set $U\subseteq M$,
then $f_1=f_2$.
\item[{\rm(b)}]
If $E$ and $F$ are real locally convex spaces,
$U\subseteq E_{\mathbb C}$ an open connected subset
with $U\cap E\not=\emptyset$
and $f_j \colon U\to F_\C$
complex analytic mappings for $j\in \{1,2\}$
such that $f_1|_{E\cap U}
= f_2|_{E\cap U}$, then $f_1=f_2$.
\end{itemize}
\end{lem}
\noindent
A subset $M$ of a topological space~$X$ is called
\emph{locally closed}
if each $x\in M$ has a neighborhood $U\subseteq X$
such that $U\cap M$ is a relatively closed subset of~$U$.
For example, every closed subset of~$X$ is locally closed.
A real analytic manifold~$M$ admitting a complexification $M^*$
(with $M\subseteq M^*$) is always locally closed in~$M^*$.
In fact, for $x\in M$ and $\psi\colon U^*\to U'$
as in the above definition, the set $E\cap U'$ is relatively closed in~$U'$
(as~$E$ is closed in~$E_{\mathbb C}$)
and thus $M\cap U^*=\psi^{-1}(E\cap U')$
is relatively closed in~$U^*$.\\[3mm]
Complexifications were already defined in the introduction.
In Section~\ref{sec-exists}, we shall find it
convenient to loosen
the concept of a complexification (by allowing passage
to diffeomorphic copies of complexifications in
the earlier sense):\footnote{Outside Section~\ref{sec-exists},
the earlier definition will be retained to simplify the notation
(without loss of\linebreak
\hspace*{-1.4mm}mathematical substance).}
\begin{defn}\label{cxgeneral}
Let $\Omega$ be a real analytic manifold modeled on a real
locally convex space~$E$. A \emph{complexification} of $\Omega$ is a
pair $(\Omega^*,\varphi)$ such that 
  \begin{itemize}
   \item $\Omega^*$ is a complex analytic manifold modeled on~$E_\C$,
   \item $\varphi \colon \Omega \rightarrow \varphi (\Omega)
   \subseteq \Omega^*$ is a real analytic
   diffeomorphism
   onto a real analytic submanifold $\varphi (\Omega)$ of $\Omega^*$,
   \item for each $x\in \Omega$,
   there are an open neighborhood $U^*$ of $x$ in $\Omega^*$ and a
   complex analytic
   diffeomorphism $U^*\to U'$ onto an open subset $U'\subseteq E_{\mathbb C}$
   which takes $\varphi(\Omega)\cap U^*$ onto 
   $E\cap U'$.
  \end{itemize}
\end{defn}
\section{Uniqueness of complexifications}\label{sec-unique}
In this section, we discuss the existence
of complex analytic extensions of real
analytic mappings between real
analytic manifolds to open neighborhoods
in given complexifications.
As a consequence, we shall obtain
a proof for Theorem~B.
We begin with purely topological considerations,
which are the foundation for the extension results.
\begin{lem}\label{toplemma}
Let $X$ be a Hausdorff topological space
and $(U_i)_{i\in I}$
be a family of open subsets of~$X$.
\begin{itemize}
\item[{\rm(a)}]
If  $X$ is regular, $M$ a compact subset of~$X$
and $(U_i)_{i\in I}$ a cover of~$M$,
then there exists a cover $(W_j)_{j\in J}$
of~$M$ by open sets $W_j\subseteq X$ such that
\begin{equation}\label{meeting}
(\forall j,k\in J) \quad W_j\cap W_k\not=\emptyset \;\;\Rightarrow\;\;
(\exists i\in I)\;\; W_j\cup W_k\subseteq U_i.
\end{equation}
\item[{\rm(b)}]
If  $X$ is paracompact and $(U_i)_{i\in I}$ a cover of~$X$,
then there exists an open cover $(W_j)_{j\in J}$
of~$X$ such that \emph{(\ref{meeting})} holds.
\item[{\rm(c)}]
If $X$ is paracompact, $M$ a closed subset of~$X$ and
$(U_i)_{i\in I}$ a cover of~$M$,
then there exists a cover $(W_j)_{j\in J}$
of~$M$ by open subsets~$W_j$ of~$X$ such that \emph{(\ref{meeting})} holds.
\item[{\rm(d)}]
If~$M$ is a locally closed subset of~$X$,
every open neighborhood of~$M$ in~$X$ contains a paracompact open neighborhood
of~$M$ in~$X$ and $(U_i)_{i\in I}$ is a cover of~$M$,
then there exists a cover $(W_j)_{j\in J}$
of~$M$ by open subsets~$W_j$ of~$X$ such that \emph{(\ref{meeting})} holds.
\end{itemize}
\end{lem}
\begin{proof}
(a) For each $x\in M$, there is $i_x\in I$ such that $x\in U_{i_x}$.
There is a closed neighborhood~$A_x$ of~$x$ in~$X$ such that
$A_x\subseteq U_{i_x}$,
by regularity of~$X$. By compactness of~$M$, we find a finite
subset $Z \subseteq M$
such that $M\subseteq \bigcup_{z\in Z}A_z^0$,
where $A_z^0$ means the interior of $A_z$ in~$X$.
For $x\in M$, we define
\[
W_x\; := \; \bigcap_{x\in A_z^0}A_z^0
\cap \bigcap_{x\in \partial A_z}U_{i_z}
\cap \bigcap_{x\not\in A_z} (X\setminus A_z),
\]
where the indices range through all $z\in Z$ with the respective property,
and intersections over empty index sets are defined as~$X$.
There is $z(x)\in Z$ such that $x\in A_{z(x)}^0$.
Then $W_x\subseteq A_{z(x)}^0\subseteq U_{i_{z(x)}}$.
If $y\in M$ and $W_x\cap W_y\not=\emptyset$,
then $y\in A_{z(x)}$.
In fact, if this was wrong,
then $y\not\in A_{z(x)}$,
entailing that $W_y\subseteq X\setminus A_{z(x)}$.
Since $W_x\subseteq A_{z(x)}$,
we deduce that $W_x\cap W_y=\emptyset$,
a contradiction.
Now $y\in A_{z(x)}$
entails that $y\in A_{z(x)}^0$ or
$y\in \partial A_{z(x)}$.
In the first case, $W_y\subseteq A^0_{z(x)}\subseteq U_{i_{z(x)}}$;
in the the second case,
$W_y\subseteq U_{i_{z(x)}}$.
Thus $W_x\cup W_y\subseteq U_{i_{z(x)}}$. Hence $(W_x)_{x\in M}$
has the desired properties.

(b) For each $x\in X$, there is $i_x\in I$ such that $x\in U_{i_x}$.
As every paracompact space is normal and hence regular,
there is a closed neighborhood~$B_x$ of~$x$ in~$X$ such that
$B_x\subseteq U_{i_x}$.
Since~$X$
is paracompact, there is a locally finite open cover
$(V_z)_{z\in Z}$ of~$X$ subordinate to $(B_x^0)_{x\in X}$.
Thus, for each $z \in Z$, there is $x(z)\in X$ such that
$V_z\subseteq B_{x(z)}^0$.
For the closure $A_z:= \overline{V_z}$
in~$X$, we get 
$A_z\subseteq B_{x(z)}\subseteq U_{i_{x(z)}}$.
For $x\in X$, define
\begin{equation}\label{intersect}
W_x\; := \; \bigcap_{x\in V_z } V_z
\cap \bigcap_{x\in \partial V_z}U_{i_{x(z)}}
\cap \bigcap_{x\not\in A_z} (X\setminus A_z).
\end{equation}
Note that, since the family $(V_z)_{z\in Z}$
is locally finite, also the family $(A_z)_{z\in Z}$
of the closures is locally finite \cite[Theorem 1.1.13]{Engelking1989}.
Therefore the first two intersections
in (\ref{intersect}) are finite intersections.
Moreover,
$\bigcup_{x\not\in A_z} A_z$
is closed as a locally finite union of closed sets
(see \cite[Corollary 1.1.12]{Engelking1989}).
Therefore its complement
$\bigcap_{x\not\in A_z} (X\setminus A_z)$
is open and hence~$W_x$
is an open neighborhood of~$x$.
Choose $z\in Z$ such that $x\in V_z$.
We now see as in the proof of~(a)
that $W_x\cap W_y\not=\emptyset$ implies
$W_x\cup W_y\subseteq U_{i_{x(z)}}$.

(c) Let ${\mathcal V}$ be the open cover of~$X$
obtained by joining the sets $(U_i)_{i\in I}$ and the open set $X\setminus M$.
Applying (b) to ${\mathcal V}$,
we find an open cover $(W_j)_{j\in J'}$
of~$X$ such that $W_j\cap W_k\not=\emptyset$
for $j,k\in J'$
implies that $W_j\cup W_k  \subseteq U$ for some set $U$ in ${\mathcal V}$.
Note that $U\not= X\setminus M$ and thus
$U=U_i$ for some $i\in I$
if we assume that $W_j$ and $W_k$ meet~$M$.
Thus, if we define
\[
J:=\{j\in J'\colon W_j\cap M \not=\emptyset\},
\]
then $(W_j)_{j\in J}$ has the desired properties.

(d) Since $M$ is locally closed, each $x\in M$ has an open neighborhood
$P_x$ in~$X$ such that $M\cap P_x$ is relatively closed
in~$P_x$. Thus $P_x\setminus M$ is open in $X$.
The set $\bigcup_{x\in M} P_x$ is an open neighborhood
of~$M$ in~$X$.
By hypothesis, there exists a paracompact open neighborhood $P\subseteq
\bigcup_{x\in M} P_x$ of~$M$ in~$X$.
After replacing~$P_x$ with $P_x\cap P$,
we may assume that $P=\bigcup_{x\in M} P_x$.
Let ${\mathcal V}$ be the open cover of~$P$
obtained by joining the sets $(P\cap U_i)_{i\in I}$
and the open set $P\setminus M =
\bigcup_{x\in M} (P_x\setminus M)$.
Applying (b) to $P$ and ${\mathcal V}$ in place of~$X$ and $(U_i)_{i\in I}$,
we find an open cover $(W_j)_{j\in J'}$
of~$P$ such that $W_j\cap W_k\not=\emptyset$
for $j,k\in J'$
implies that $W_j\cup W_k  \subseteq U$ for some set $U$ in ${\mathcal V}$.
Note that $U\not= P\setminus M$ and thus
$U=P\cap U_i\subseteq U_i$ for some $i\in I$
if we assume that $W_j$ and $W_k$ meet~$M$.
Thus, if we define
$J:=\{j\in J'\colon W_j\cap M \not=\emptyset\}$,
then $(W_j)_{j\in J}$ has the desired properties.
\end{proof}
\begin{lem}\label{extensions}
Let $M$ and $N$ be real analytic manifolds modeled on locally convex spaces
and $f\colon M\to N$ be a real analytic map.
Let $M^*$ and $N^*$
be complexifications of~$M$ and~$N$,
respectively,
such that $M\subseteq M^*$ and $N\subseteq N^*$.
\begin{itemize}
\item[{\rm(a)}]
If $K\subseteq M$ is compact and
$M^*$ is a regular topological
space, then there exists a complex
analytic map $g \colon U\to N^*$,
defined on an open neighborhood~$U$ of~$K$ in~$M^*$,
such that $g|_{U\cap M}=f|_{U\cap M}$.
\item[{\rm(b)}]
If every open neighborhood of~$M$ in~$M^*$
contains a paracompact open neighborhood of~$M$ in~$M^*$,
then there exists a complex
analytic extension $g \colon U\to N^*$ of~$f$,
defined on an open neighborhood~$U$ of~$M$ in~$M^*$.
\item[{\rm(c)}]
If $M^*$ is paracompact and $M$ is closed in~$M^*$,
then there exists a complex
analytic extension $g \colon U\to N^*$ of~$f$,
defined on an open neighborhood~$U$ of~$M$ in~$M^*$.
\item[{\rm(d)}]
If $g_j\colon U_j\to N^*$
are complex analytic extensions of~$f$ for $j\in\{1,2\}$,
defined on open neighborhoods~$U_j$ of~$M$ in~$M^*$,
then the union~$U$ of all connected components~$C$ of~$U_1\cap U_2$ such that
$C\cap M\not=\emptyset$
is an open neighborhood of~$M$ in~$M^*$,
and $g_1|_U=g_2|_U$.
\item[{\rm(e)}]
Let $f$ be a real analytic diffeomorphism in the sense
that $f^{-1}$ exists and is real analytic.
Assume
that~$f$ admits a complex analytic extension
$g\colon U\to N^*$,
defined on an open neighborhood~$U$ of~$M$ in~$M^*$, and
that~$f^{-1}$ admits a complex analytic extension
$h\colon V\to M^*$,
defined on an open neighborhood~$V$ of~$N$ in~$N^*$.
Then there exist an open neighborhood~$W$ of~$M$ in~$U$
such that~$g(W)$ is an open subset of~$N^*$
contained in~$V$, and $g|_W\colon W\to g(W)$
is a complex analytic diffeomorphism with inverse~$h|_{g(W)}$.
\end{itemize}
\end{lem}
\begin{proof}
Let $E$ be the modeling space of~$M$ and~$F$ be the modeling space of~$N$.

(a) For $x\in K$, let $\psi_x \colon U_x^*\to U'_x$ be a chart of~$M^*$
around~$x$ and $\theta_x\colon V_x^*\to V'_x$ be a chart of~$N^*$
around~$f(x)$ such that $\psi_x(M\cap U^*_x)=E\cap U'_x$
and $\theta_x(N\cap V^*_x)=F\cap V'_x$.
After shrinking~$U_x$, we may assume that
$f(M\cap U_x^*)\subseteq V_x^*$. Thus
\begin{equation}\label{re-use}
E\cap U'_x\to F\cap V_x',\quad
y\mapsto \theta_x (f(\psi_x^{-1}(y)))
\end{equation}
is a real analytic map,
which admits a complex analytic extension
$h_x\colon P_x\to V_x'$ on some
open neighborhood $P_x$ of $E\cap U'_x$ in~$U'_x$.
Then the open sets $U_x:=\psi_x^{-1}(P_x)\subseteq M^*$,
for $x\in K$, form a cover of~$K$, and the maps
\[
g_x:= \theta_x^{-1}\circ h_x\circ \psi_x|_{U_x}\colon U_x\to V_x^*\subseteq N^*
\]
are complex analytic and satisfy $g_x|_{M\cap U_x}=
f|_{M\cap U_x}$. Using Lemma~\ref{toplemma}\,(a),
we find a cover $(W_j)_{j\in J}$ of~$K$
by open subsets of~$M^*$ such that
$W_j\cap W_k\not=\emptyset$
implies the existence of some $x(j,k)\in K$
such that $W_j\cup W_k\subseteq U_{x(j,k)}$.
After replacing the sets $W_j$ by smaller neighborhoods
of the points in $M\cap W_j$, we may assume that each~$W_j$
is connected and $M\cap W_j\not=\emptyset$.
Define
\[
g_j:=g_{x(j,j)}|_{W_j}\colon W_j\to N^*.
\]
If $W_j\cap W_k\not=\emptyset$,
we have
\begin{equation}\label{little}
g_j=g_{x(j,k)}|_{W_j}
\end{equation}
by the Identity Theorem (Lemma~\ref{idthm}),
because $W_j$ is connected and both sides of~(\ref{little})
are complex analytic extensions
of $f|_{M\cap W_j}$ (where $M\cap W_j\not=\emptyset$).
Likewise, $g_k=g_{x(j,k)}|_{W_k}$, and thus
\[
g_j|_{W_j\cap W_k}=g_{x(j,k)}|_{W_j\cap W_k}= g_k|_{W_j\cap W_k}.
\]
Now $U:=\bigcup_{j\in J}W_j$ is an open neighborhood
of~$K$ in~$M^*$. By the preceding, the map
$g\colon U\to N^*$ defined via $g(z):=g_j(z)$ if $z\in W_j$
is well defined. By construction, it is a complex analytic extension
of $f|_{M\cap U}$.

(b) Using Lemma~\ref{toplemma}\,(d) instead of Lemma~\ref{toplemma}\,(a),
we can repeat the proof of~(a) with $K:=M$.

(c) Using Lemma~\ref{toplemma}\,(c) instead of Lemma~\ref{toplemma}\,(a),
we can repeat the proof of~(a) with $K:=M$.

(d) Since each~$C$ (as in the statement) is connected and $g_1|_{M\cap C}=g_2|_{M\cap C}$
with $M\cap C\not=\emptyset$, the Identity Theorem
shows that $g_1|_C=g_2|_C$. The assertion follows.

(e) Let $W$ be the union of all connected components~$C$
of the open set $g^{-1}(V)$ such that $M\cap C\not=\emptyset$.
Then $h\circ g|_W$
defines a complex analytic map $W\to M^*$
such that $h(g(x))=x$ for all $x\in M\cap W$.
Hence $h\circ g|_W=\id_W$, by the Identity Theorem.
In particular, $W\subseteq h(V)$,
and $g|_W\colon W\to g(W)$ is a bijection with $(g|_W)^{-1}=h|_{g(W)}$.
Let~$Q$ be the union of all connected components~$D$  of $h^{-1}(W)$
such that $N\cap D\not=\emptyset$.
Arguing as before, we see that $g|_W\circ h|_Q=\id_Q$.
In particular, $Q\subseteq g(W)$.
If $C\subseteq W$ is as above, then $g(C)$ is a connected subset of~$N^*$
such that $N \cap g(C)\not=\emptyset$.
Since $h(g(C))=C\subseteq W$, we have $g(C)\subseteq h^{-1}(W)$
and thus $g(C)\subseteq D$ for some~$D$ as above.
Hence $g(C)\subseteq Q$, entailing that $C=h(g(C))\subseteq h(Q)$.
We deduce that $W=h(Q)$. Thus $h|_Q\colon Q\to W$
is a bijection with inverse~$g|_W$.
\end{proof}
\noindent
{\bf Proof of Theorem B.}
(a) By Lemma~\ref{extensions}\,(c), the real analytic map $\id_M\colon M\to M$
extends to complex analytic mappings
$g\colon U\to M_2^*$
and $h\colon V\to M_1^*$,
defined on open neighborhoods $U\subseteq M_1^*$ and $V\subseteq M_2^*$,
respectively, of~$M$.
By Lemma~\ref{extensions}\,(e), after shrinking~$U$ and~$V$,
the map~$g$ is a complex analytic diffeomorphism
from~$U$ onto~$V$ (with inverse~$h$).

(b) Applying Lemma~\ref{extensions}\,(a)
to the identity map $\id_M\colon M\to M$,
we find open neighborhoods $U$ and $V$ of~$K$
in~$M_1^*$ and $M_2^*$, respectively,
and complex analytic mappings
$g\colon U\to M_2^*$ and $h\colon V\to M_1^*$
such that $g|_{M\cap U}=\id_{M\cap U}$
and
$h|_{M\cap V}=\id_{M\cap V}$.
Then $N:=(M\cap U)\cap (M\cap V)$
is an open neighborhood of~$K$ in~$M$,
and both~$M_1^*$ and $M_2^*$
are complexifications of~$N$.
By Lemma~\ref{extensions}\,(e), after shrinking $U$ and~$V$,
the map $g$ is a complex analytic diffeomorphism
from $U$ onto~$V$ (with inverse~$h$).\Punkt
\begin{prop}\label{specialize}
Let $M^*$ be a complexification
of a real analytic manifold~$M$.
\begin{itemize}
\item[{\normalfont (a)}]
If every open neighborhood of~$M$ in~$M^*$ contains
a paracompact open neighborhood of~$M$ in~$M^*$,
then there exists a paracompact open neighborhood~$U$
of~$M$ in~$M^*$ such that~$M$ is closed in~$U$.
\item[{\normalfont (b)}]
If~$M^*$ is paracompact and~$M$ is closed in~$M^*$,
then there exists an open neighborhood~$U$ of~$M$
in~$M^*$ and an anti-holomorphic map
$\sigma\colon U\to U$ such that $\sigma\circ\sigma=\id_U$ and
$M=\{z\in U\colon \sigma(z)=z\}$.
\item[{\normalfont (c)}]
If $K\subseteq M$ is a non-empty
compact set and $M^*$ is regular as a topological
space, then there exists an open neighborhood~$U$ of~$K$
in~$M^*$ and an anti-holomorphic map
$\sigma\colon U\to U$ such that $\sigma\circ\sigma=\id_U$ and
$M\cap U=\{z\in U\colon \sigma(z)=z\}$.
In particular, $M\cap U$ is closed in~$U$.
\end{itemize}
\end{prop}
\begin{proof}
(a) Let $E$ be the modeling space of~$M$
and $\tau\colon E_\C\to E_\C$ be complex
conjugation, $\tau(x+iy):=x-iy$
for $x,y\in E$.
By Remark~\ref{complex-op}
and Lemma~\ref{extensions}\,(b), $\id_M$
admits
a complex analytic extension $g\colon V\to (M^*)_{\op}$,
defined on an open neighborhood~$V$ of~$M$ in~$M^*$.
For $x\in M$, let $\psi_x\colon U^*_x\to U'_x$
be a chart for~$M^*$ such that $\psi_x(M\cap U_x^*)=E\cap U'_x$
and $U_x^*\subseteq V$.
Then $\tau\circ \psi_x$ is a chart for~$(M^*)_{\op}$.
Let $Q_x\subseteq U_x^*$ be a connected neighborhood of~$x$
such that $g(Q_x)\subseteq U_x^*$.
Then $P_x:=\psi_x(Q_x)$ is a connected open
subset of~$E_\C$ such that $\psi_x(x)\in E\cap P_x$.
Because
$\tau\circ \psi_x\circ g\circ \psi_x^{-1}|_{P_x}$
and $\id_{P_x}$ are complex analytic maps
$P_x\to E_\C$ which coincide with $\id_{E\cap P_x}$
on~$E\cap P_x$,
the Identity Theorem shows that
$\tau\circ \psi_x\circ g\circ \psi_x^{-1}|_{P_x}=\id_{P_x}$
and thus $\psi_x\circ g\circ \psi_x^{-1}|_{P_x}=\tau|_{P_x}$.
As the latter map moves each vector in $P_x\setminus E$,
we deduce that $g(z)\not=z$ for all
$z\in Q_x\setminus M$. Hence, after replacing~$V$
with its open subset $Q:=\bigcup_{x\in M}Q_x$,
we may assume that $g(z)\not=z$ for all $z\in V\setminus M$.
By
Lemma~\ref{extensions}\,(e),
there is an open neighborhood~$W$ of~$M$ in~$V$ such that
$g(W)$ is open and $g\colon W\to g(W)\subseteq (M^*)_{\op}$
is a complex analytic diffeomorphism,
with inverse~$g|_{g(W)}$.
Henceforth, we may consider~$g$ as an anti-holomorphic
diffeomorphism $W\to g(W)\subseteq M^*$.
Now $P:=W\cap g(W)$ is an open neighborhood
of~$M$ in~$M^*$ and $\sigma:= g|_P$
is an antiholomorphic involution
$P\to P$ such that $M\subseteq \{z\in P\colon \sigma(z)=z\}$
and actually $M=\{z\in P\colon \sigma(z)=z\}$
(as we ensured by passing to~$Q$
that~$\sigma$ moves all points outside~$M$).
In particular, $M$ is closed in~$P$.
By hypothesis, there exists a paracompact open neighborhood~$U$
of~$M$ in~$P$. By the preceding, $M$~is closed in~$U$.

(b) Using Lemma~\ref{extensions}\,(c)
instead of Lemma~\ref{extensions}\,(b), we can repeat the proof of~(a)
to find an open neighbourhood~$P$ of~$M$ in~$M^*$
and an anti-holomophic involution $\sigma\colon P\to P$
with fixed point set~$M$.

(c) Using Lemma~\ref{extensions}\,(a)
instead of Lemma~\ref{extensions}\,(b),
we obtain $g\colon V\to (M^*)_{\op}$
on an open neighborhood~$V$ of~$K$ in~$M^*$.
We can now continue as in the proof
of~(a), replacing each occurence of~``$M$''
with ``$V\cap M$'' and taking
$U:=P$ at the end.
\end{proof}
\section{Existence of complexifications around a\\
compact set}\label{sec-exists}

In this section, we prove Theorem A. The proof splits into a proposition,
where the complexification is constructed, and a corollary. The latter
establishes\linebreak
regularity of the complexification. 

\begin{prop}\label{prop: THMA1}
Let $M$ be a real analytic manifold modeled on a locally convex space~$E$.
Assume that the topological space~$M$ is regular and~$E_\C$ is normal.
For each compact subset~$K$ of~$M$, there is an open neighbourhood~$\Omega$
of~$K$ in~$M$ such that the open submanifold~$\Omega$ admits a
complexification $(\Omega^*,\varphi)$.   
\end{prop}

\begin{proof}
We construct the neighborhood $\Omega$ of $K$ as a union of finitely many
chart domains:
% The space $E_\C$ is a normal topological space,
% whence it is regular.
% Identify~$E$ with the subspace $E \times \set{0} \subseteq E_\C$. Then~$E$
% is a regular topological space as a subspace of $E_\C$.\\
For $x \in K$, choose a manifold chart
$\varphi_x \colon M \supseteq T_x' \rightarrow T_x$ of $M$ with
$x \in T_x'$. Since $E$ is regular as a topological vector space and $M$ is regular by assumption, 
\cite[Proposition 1.5.5]{Engelking1989} allows us to choose an open
neighborhood~$W_x'$ of~$x$
with closure $\overline{W}_x' \subseteq T_x'$
in~$M$.
Then $\varphi_x (\overline{W}_x')$ is a closed (with respect to the subspace
topology) neighborhood of $\varphi_x (x)$ in $T_x$. Moreover, we can choose an
open neighborhood $W_x$ of $\varphi_x (x)$ with
closure $\overline{W}_x \subseteq T_x$ in~$E$. Define the open sets 
  \begin{displaymath}
   U_x \coloneq \varphi_x (W_x') \cap W_x \text{ and } U_x' \coloneq
  \varphi_x^{-1} (U_x) = W_x' \cap \varphi_x^{-1} (W_x)
  \end{displaymath}
and observe that the set $\overline{U}_x \subseteq \overline{W}_x$ is
closed in~$E$ and $\overline{U}_x' \subseteq \overline{W}_x'$ is closed in~$M$.
Analogously, we construct open neighborhoods $V_x'$ and $V_x$ with
$\varphi_x (V_x') = V_x$, such that $\overline{V}_x' \subseteq U_x'$
and $\overline{V}_x \subseteq U_x$. The open sets
$(V_x')_{x \in K}$
cover~$K$. Since~$K$ is compact, we obtain finite families of open subsets
$(V_i')_{i \in I}, (U_i')_{i \in I}$ and $(T_i')_{i \in I}$ of~$M$ with the
following properties for each $i \in I$:\\
 There is a map $\varphi_i \colon T_i' \rightarrow T_i \subseteq E$ such that
$(T_i',\varphi_i)$ is a manifold chart for~$M$. The sets are ordered via
$\overline{V}_i' \subseteq U_i' \subseteq \overline{U}_i' \subseteq T_i'$
and the sets $\overline{V}_i',\overline{U}_i'$ are closed in~$M$. The compact
set~$K$ is contained in $\Omega \coloneq \bigcup_{i \in I} V_i' \subseteq M$.
Furthermore, the sets $\varphi_i (\overline{V}_i')$ and
$\varphi_i (\overline{U}_i')$ are closed in~$E$. Notice that~$\Omega$ is an
open real analytic submanifold of~$M$, such that the charts
$(V_i',\varphi_i|_{V_i'})_{i \in I}$ form an atlas for $\Omega$.\bigskip

 We now construct a complexification for the open submanifold $\Omega$. To
shorten the notation, we define the following sets for $i,j \in I$: 
 \begin{align*}
  U_i \coloneq \varphi_i (U_i'), \quad \quad V_i \coloneq \varphi_i (V_i'),
  \quad \quad T_i \coloneq \varphi_i (T_i') \subseteq E, \\
  U_{i,j} \coloneq \varphi_i (U_i'\cap U_j'), \quad V_{i,j}
  \coloneq \varphi_i (V_i'\cap V_j'), \quad  T_{i,j} \coloneq \varphi_i (T_i' \cap T_j').
 \end{align*}
The change of charts $\varphi_j \circ \varphi_i^{-1}|_{T_{i,j}}
\colon T_{i,j} \rightarrow T_{j,i}$ is a real analytic diffeomorphism. Hence,
by Lemma~\ref{extensions}\,(e),
there are open neighborhoods $T_{i,j}^*$ of $T_{i,j}$ and $T_{j,i}^*$ of
$T_{j,i}$ in~$E_\C$ together with a complex analytic diffeomorphism
$\psi_{i,j} \colon T_{i,j}^* \rightarrow T_{j.i}^*$ which extends 
$\varphi_j \circ \varphi_i^{-1}|_{T_{i,j}}$.
Adjusting choices, we can achieve that $T_{i,j}^* \cap E = T_{i,j}$ and
$\psi_{i,i} = \id_{T_{i,i}}$ for all $i,j \in I$. Without loss of generality,
$T_{i,j}^*$ is the empty set if $T_{i,j}$ is empty and 
    \begin{equation}\label{eq: iso:inv}
     \psi_{j,i} = \psi_{i,j}^{-1} \quad \text{ for each pair } (i,j) \in I \times I.
    \end{equation}
 For $(i,j) \in I \times I$, the open set $U_{i,j}$ satisfies
$\overline{U}_{i,j} \subseteq \overline{U}_i \cap \overline{U}_j
\subseteq T_{i,j}$. Here the closure in~$T_{i,j}$ coincides with the closure
in~$E_\C$. Hence $\overline{U}_{i,j} \times \set{0}$ is a closed subset
of~$T_{i,j}^*$. By normality of~$E_\C$, we can choose an open
neighborhood~$O_{i,j}$ of $\overline{U}_{i,j} \times \set{0}$ whose closure
with respect to~$E_\C$ satisfies $\overline{O}_{i,j} \subseteq T_{i,j}^*$.
The open subset $U_{i,j}^* \coloneq O_{i,j} \cap (U_{i,j} \times E)$
of~$T_{i,j}^*$ satisfies 
    \begin{equation}\label{eq: U:E}
    \overline{U}_{i,j}^* \subseteq T_{i,j}^* , \quad  U_{i,j}^* \cap E = U_{i,j}
    \quad \text{ and } \quad \overline{U}_{i,j}^* \cap E = \overline{U}_{i,j}.
    \end{equation}
 Shrinking the sets $U_{i,j}^*$, we can achieve that $\psi_{i,j} (U_{i,j}^*)
= U_{j,i}^*$, since by construction we have
$\psi_{i,j} (U_{i,j} \times \set{0})
= \varphi_j \varphi_i^{-1} (U_{i,j}) \times \set{0}
= U_{j,i} \times \set{0}$.\\
 The set $\overline{V}_i \cap \varphi_i \varphi_j^{-1} (\overline{V}_j
\cap \overline{U}_{j,i})$ is closed in~$U_{i,j}$. As above, we can choose an
open neighborhood~$W_{i,j}^*$ of
$(\overline{V}_i \cap \varphi_i \varphi_j^{-1} (\overline{V}_j \cap
\overline{U}_{j,i})) \times \set{0}) $ whose closure (with respect to~$E_\C$)
is contained in~$U_{i,j}^*$. Adjusting choices, we obtain
$\psi_{i,j}(W_{i,j}^*) = W_{j,i}^*$ for each pair $(i,j) \in I\times I$.\\
 Observe that
$\psi_{j,i} ((\overline{V}_j \cap \overline{U}_{j,i})\times \set{0})$ is
closed in~$E_\C$, since it is relatively closed in~$T_{i,j}$ and contained
in the $E_\C$-closed set
$\overline{U}_{i,j} \times \set{0} \subseteq T_{i,j}$. Thus
$(\overline{V}_i \times \set{0}) \setminus W_{i,j}^*$ and
$\psi_{j,i} ((\overline{V}_j \cap
\overline{U}_{j,i})\times \set{0})\setminus W_{i,j}^*$ are disjoint closed
subsets of $E_\C$. As~$E_\C$ is a normal space, there are disjoint open
subsets~$A_{i,j}^*$ and~$B_{i,j}^*$ of~$E_\C$ such that 
  \begin{equation}\label{eq: AB}
   \overline{V}_i \subseteq A_{i,j}^* \cup W_{i,j}^* \quad \mbox{and}\quad
   \quad \psi_{j,i} (\overline{V}_j \cap \overline{U}_{j,i}) \subseteq W_{i,j}^* \cup B_{i,j}^*.
  \end{equation}
 An argument analogous to the one which produced \eqref{eq: U:E} yields an open subset $A_i^*$ of $E_\C$ with the following properties: 
  \begin{align} 
      \label{eq: A1} A_i^* \cap E = V_i , \quad \quad \overline{A}_i^* \cap E = \overline{V}_i, \;\;\mbox{ and} \\
      \label{eq: A2} A_i^* \subseteq \bigcap_{j \in I,\, T_{i,j}^* \neq \emptyset}A_{i,j}^* \cup W_{i,j}^*.  
      \end{align}
 By construction, $\overline{A}_i^* \cap \overline{U}_{i,j}^*$ is a closed
subset of $\overline{U}_{i,j}^* \subseteq T_{i,j}^*$, whence the
diffeomorphism $\psi_{i,j}$ satisfies
$\overline{\psi_{i,j} (A_i^* \cap U_{i,j}^*)} =
\psi_{i,j} (\overline{A}_i^* \cap \overline{U}_{i,j}^*)$.
Combine \eqref{eq: U:E} and \eqref{eq: A1} to obtain
 \begin{equation}\label{eq: inc}
     \overline{\psi_{i,j} (A_i^* \cap U_{i,j}^*)} \cap E
  = \psi_{i,j} (\overline{A}_i^* \cap \overline{U}_{i,j}^* \cap E) \subseteq
 \psi_{i,j} ((\overline{V}_i \cap \overline{U}_{i,j}) \times \set{0}).
 \end{equation}
 Since $I$ is a finite set, we may choose for each $i \in I$ and $x\in U_i$ an
open neighborhood~$U_{i,x}^*$ in~$E_\C$ such that the following properties are satisfied: 
 \begin{enumerate}
   \item[(a)] For each $j \in I$ with $x \in U_{i,j}$, we have
   $U_{i,x}^* \subseteq U_{i,j}^*$.
   \item[(b)] If $x \in \psi_{j,i} ((\overline{V}_j \cap
   \overline{U}_{j,i}) \times \set{0})$, then
   $U_{i,x}^* \subseteq W_{i,j}^* \cup B_{i,j}^*$ (cf.\ \eqref{eq: AB}).
   \item[(c)] For each $j \in I$ with
   $\varphi_i^{-1} (x) \not \in \overline{V}_j'$, the intersection
   $U_{i,x}^* \cap \psi_{j,i} (A_i^* \cap U_{j,i}^*)$ is empty.
   \item[(d)] For each $(j,k) \in I \times I$ with $x \in U_{i,j} \cap U_{i,k}$,
   i.e.\ $\varphi_i^{-1} (x) \in U_i' \cap U_j' \cap U_k'$, we have
	 \begin{displaymath}
	  U_{i,x}^* \subseteq \psi_{j,i} (U_{j,i}^* \cap U_{j,k}^*) \cap
          \psi_{k,i} (U_{k,i}^* \cap U_{k,j}^*),
	 \end{displaymath}
     such that the cocycle condition
$\psi_{i,j}|_{U_{i,x}^*} = \psi_{k,j} \circ \psi_{i,k}|_{U_{i,x}^*}$
holds.\footnote{This is possible as the identity already holds on
$(U_{i,j} \cap  U_{i,k}) \times \set{0}$.}
  \end{enumerate}
 It is possible to choose neighborhoods with property (c) because of the
following observations: If $U_{j,i}^* = \emptyset$, then the property is
trivially satisfied. For $j \in I$ with $U_{j,i}^* \neq \emptyset$,
the condition $\varphi_i^{-1} (x) \not \in \overline{V}_j'$ implies
$x \not \in \overline{\psi_{j,i} (A_i^* \cap U_{j,i}^*)}$ by \eqref{eq: inc}.\\
 Define the set $U_i^* \coloneq \bigcup_{x \in U_i} U_{i,x}^*$ and observe
that $\overline{V}_i \times \set{0}$ is a closed set contained in~$U_i^*$.
By \eqref{eq: A1}, there is an open neighborhood $V_i^*$ of $V_i$ which is
contained in $A_i^* \cap U_i^*$. As~$E_\C$ is normal, we can choose~$V_i^*$
such that its closure with respect to~$E_\C$ is contained in~$U_i^*$.
The identities \eqref{eq: A1} then yield $V_i^* \cap E = V_i$ and
$\overline{V}_i^* \cap E = \overline{V}_i$. Define the following sets for
$i,j,k \in I$: 
  \begin{equation}\label{eq: newset}
   V_{i,j}^* \coloneq V_i^* \cap \psi_{j,i} (V_j^* \cap U_{j,i}^*)\quad
   \text{ and }\quad V_{i,j,k}^* \coloneq V_{i,j}^* \cap V_{i,k}^*.
  \end{equation}
 By construction, $V_{i,j}^* \subseteq U_{i,j}^*$ is satisfied, since
$\psi_{j,i}$ maps $U_{j,i}^*$ to $U_{i,j}^*$. In particular, by
\eqref{eq: newset} $\psi_{i,j}$ restricts to a map
$\psi_{i,j}|_{V_{i,j}^*}^{V_{j,i}^*}$, which is a complex analytic
diffeomorphism. Finally, we remark that $V_{i,i}^* = V_i^*$.\\
 Each point $y\in V_{i,j,k}^*$ is contained in an open set~$U_{i,x}^*$ for
some $x \in U_i$. From \eqref{eq: newset}, we infer
$y \in U_{i,x}^* \cap \psi_{j,i} (V_j^* \cap U_{j,i}^*)
\cap \psi_{k,i} (V_k^* \cap U_{k,i}^*)$. Thus, by choice of~$V_j^*$
and~$V_k^*$, the intersection
$U_{i,x}^* \cap \psi_{j,i} (A_j^* \cap U_{j,i}^*) \cap
\psi_{k,i} (A_k^* \cap U_{k,i}^*)$ is not empty.
We conclude from
property~(c) for~$U_{i,x}^*$ the following condition:
  \begin{displaymath}
   x \in \varphi_i (\overline{V}_j' \cap V_i') \cap
   \varphi_i (\overline{V}_k' \cap V_i') \subseteq
   \varphi_i (U_i' \cap U_j') \cap \varphi_i (U_i' \cap U_k')
   = U_{i,j} \cap U_{i,k}. 
  \end{displaymath}
 Hence property (d) for $U_{i,x}^*$ yields $\psi_{k,j} \circ \psi_{i,k} (y)
= \psi_{i,j} (y)$. The point $z \coloneq \psi_{i,j} (y)$ is contained in
$V_{j,i}^*$ since $\psi_{i,j}$ maps $V_{i,j}^*$ into this set. In particular,
$z \in V_{j}^*$ holds. Furthermore, by definition of~$V_{i,j,k}^*$, we obtain
$\psi_{i,k} (y) \in V_k^* \cap  U_{k,i}^*$ (cf.\ \eqref{eq: newset}). Hence
we deduce $z = \psi_{k,j} \circ \psi_{i,k} (y) \subseteq
\psi_{k,j} (U_{k,i}^* \cap V_k^*)$.
Summing up, we derive $z = \psi_{i,j} (y) \in V_{j,i,k}^*$.\\
 As $y \in V_{i,j,k}^*$ was arbitrary, $\psi_{i,j}$ maps $V_{i,j,k}^*$ into
$V_{j,i,k}^*$. An analogous argument yields
$\psi_{j,i} (V_{j,i,k}^*) \subseteq V_{i,j,k}^*$. From \eqref{eq: iso:inv},
we deduce that $\psi_{i,j}|_{V_{i,j,k}^*}^{V_{j,i,k}^*}$ is an analytic
diffeomorphism whose inverse is $\psi_{j,i}|^{V_{i,j,k}^*}_{V_{j,i,k}^*}$;
moreover, the identity
$\psi_{i,j}|_{V_{i,j,k}^*} = \psi_{k,j} \circ \psi_{i,k}|_{V_{i,j,k}^*}$
holds.

Consider the disjoint union (topological sum)
$\tilde{\Omega} \coloneq \bigsqcup_{i \in I} V_i^*$ and recall
$V_i^* = V_{i,i}^*$. We declare a relation ``$\sim$'' on~$\tilde{\Omega}$.
For two points $x \in V_i^*$ and $y \in V_j^*$, define: 
  \begin{displaymath} 
    x\sim y\ \Longleftrightarrow\ x \in V_{i,j}^*,\ y \in V_{j,i}^*
    \text{ and } y = \psi_{i,j} (x).
  \end{displaymath}
 From the construction above, it is clear that ``$\sim$'' is an equivalence
relation on~$\tilde{\Omega}$. Define~$\Omega^*$ to be the quotient space
$\tilde{\Omega}/{\sim}$.
Denote the associated quotient map by
$q \colon \tilde{\Omega} \rightarrow \Omega^*$. Then the composition
$\kappa_i \coloneq q \circ (V_i^* \hookrightarrow \tilde{\Omega})$ (with the
natural embedding) is injective. The pairs $(V_i^* , \kappa_i)_{i \in I}$ form
a family of complex charts for~$\Omega^*$. Indeed this family is an atlas
for~$\Omega^*$ whose change of chart maps for $(i,j) \in I\times I$ are the
complex analytic maps~$\psi_{i,j}$. Hence~$\Omega^*$ with this atlas becomes
a (possibly non-Hausdorff) complex analytic  manifold. 
\paragraph{Claim:} $\Omega^*$ is a Hausdorff topological space.\\
If the claim is true then $\Omega^*$ is a Hausdorff complex analytic
manifold. Furthermore, the map
$\bigsqcup_{i \in I} \varphi_i|_{V_i'} \colon
\bigsqcup_{i \in I} V_i' \rightarrow \tilde{\Omega}$ is a real analytic
embedding whose image is contained in $\tilde{\Omega} \cap E$ (identifying~$E$
with the real subspace $E \times \set{0}$). The change of chart maps
for~$\Omega^*$ are complex analytic extensions~$\psi_{i,j}$ of the changes of
charts $\varphi_j \circ \varphi_i^{-1}$. Hence
$\bigsqcup_{i \in I} \varphi_i|_{V_i'}$ factors
to a real analytic embedding
$\varphi \colon \Omega \rightarrow \Omega^*$ whose image is a real analytic
submanifold of~$\Omega^*$. By construction, the pair $(\Omega^* , \varphi)$
is a complexification of~$\Omega$. Hence the proof will be complete,
if we can verify that~$\Omega^*$ is a Hausdorff space.\bigskip 

 \textbf{Proof of the claim: } As a first step, we prove
$\overline{V}_{i,j}^* \subseteq U_{i,j}^*$ for $(i,j) \in I \times I$. We
already know $\overline{W}_{i,j}^* \subseteq U_{i,j}^*$, whence it suffices
to prove $V_{i,j}^* \subseteq W_{i,j}$ for $(i,j) \in I \times I$. Moreover,
it suffices to consider the case $T_{i,j} \neq \emptyset$, since
otherwise~$V_{i,j}^*$ is empty. Let $y \in V_{i,j}^*$. By construction,
$V_{i,j}^* \subseteq V_i^* \subseteq U_i^*$ holds and there is some
$x \in U_i$ with $y \in U_{i,x}^*$. If~$x$ is not contained in
$\psi_{j,i} ((\overline{V}_j \cap \overline{U}_{j,i}) \times \set{0})$,
we deduce that $\varphi_i^{-1} (x)$ is not contained in $\overline{V}_j'$.
Hence property~(c) for~$U_{i,x}^*$ implies
$y \not \in \psi_{j,i} (A_i^*\cap U_{j,i}^*)$. As $V_i^* \subseteq A_{i}^*$
holds, we obtain $y \not \in \psi_{j,i} (V_i^* \cap U_{j,i}^*)$. However,
this contradicts $y \in V_{i,j}^* \subseteq \psi_{j,i} (V_{j,i}^*)$
(cf.\ \eqref{eq: newset}). 
Therefore we must have
$x \in \psi_{j,i} ((\overline{V}_j \cap \overline{U}_{j,i}) \times \set{0})$.\\
 Then property (b) for $U_{i,x}^*$ implies $y \in W_{i,j}^* \cup B_{i,j}^*$.
On the other hand,
$y \in V_i^* \subseteq A_i^* \subseteq A_{i,j}^* \cup W_{i,j}^*$
by \eqref{eq: A2} and the definition of~$V_i^*$.
Since $A_{i,j}^*$ and $B_{i,j}^*$ are disjoint, we obtain $y \in W_{i,j}^*$.
Thus $V_{i,j}^* \subseteq W_{i,j}^*$ is satisfied.\\
 To prove the Hausdorff property, let $x', y' \in \Omega^*$ be two distinct
points. We choose $i,j \in I$ and points $x \in V_i^*$ and $y \in V_j^*$
with $q( x) =x'$ and $q(y) = y'$.\\
 We have to construct neighborhoods of~$x$ in~$V_{i}^*$ and of~$y$ in~$V_j^*$
which contain no equivalent points.
To reach a contradiction, suppose to the contrary
that this was not possible. We could then find nets
$(x_\alpha)_{\alpha \in A}$ and $(y_\alpha)_{\alpha \in A}$ indexed by a
set~$A$ with $x_\alpha \rightarrow x$ and $y_\alpha \rightarrow y$ such that,
for each $\alpha \in A$, we have $x_\alpha \in V_{i,j}^*$,
$y_{\alpha} \in V_{j,i}^*$ and $x_\alpha = \psi_{j,i} (y_\alpha )$.
Observe that this implies $x \in \overline{V}_{i,j}^* \subseteq U_{i,j}^*$
and $y \in \overline{V}_{j,i}^* \subseteq U_{j,i}^*$. The mapping
$\psi_{j,i}|_{U_{j,i}^*}^{U_{i,j}^*}$ is continuous, whence
$x = \psi_{j,i} (y)$ follows. We deduce that $y \in V_j^* \cap U_{j,i}^*$
and thus $x \in V_i^* \cap \psi_{j,i} (V_j^* \cap U_{j,i}^*) = V_{i,j}^*$.
Analogously, one derives $y \in V_{j,i}^*$.
Now $x' = y'$ follows from the definition of~``$\sim$''.
This contradicts our choices of~$x'$ and~$y'$, 
whence $\Omega^*$ must be a Hausdorff space.    
\end{proof}

\begin{cor}\label{cor: regular}
 The manifold $\Omega^*$ constructed in Proposition {\rm\ref{prop: THMA1}} is a
regular topological space. 
\end{cor}

\begin{proof}
 Let $x \in \Omega^*$ and consider a closed set $C \subseteq \Omega^*$ with
$x \not \in C$. We construct open disjoint neighborhoods of~$x$ and~$C$,
respectively. Let $q \colon \tilde{\Omega} \rightarrow \Omega^*$ be the
quotient map and define $C_i \coloneq q^{-1} (C) \cap V_i^*$. Then~$C_i$ is
relatively closed in~$V_i^*$. Furthermore, there is a non-empty subset
$J \subseteq I$ such that $q^{-1} (x) \cap V_i^*$ is non-empty if and only if
$i \in J$. Since $q|_{V_i^*}$ is injective, there is a unique
$x_i \in q^{-1} (x)$ for each $i \in J$. Observe that this implies for
$i \in J,\ j \in I$ that $x_i$ is contained in $V_{i,j}^*$ if and only if
$j \in J$. \\
 The space $E_\C$ is normal, hence a regular topological space. The subsets
$V_i^* \subseteq E_\C$ are thus regular spaces by
\cite[Theorem 2.1.6]{Engelking1989}. We construct a family of pairwise
disjoint open sets $X_i, Y_i \subseteq V_i^*$ for each $i \in I$ as follows: \bigskip

  \textbf{Step 1:} Let $i \in J$. The space $V_i^*$ is regular, whence there
are disjoint open subsets $X_{i}, Y_{i} \subseteq V_i^*$ such that
$x_i \in X_i$ and $C_i \subseteq Y_i$. \bigskip

  \textbf{Step 2:} Consider $j \in I$ such that $x$ is not contained in
$\overline{q(V_j^*)}$. In other words, $x_i$ is not contained in
$\overline{V}_{i,j}^*$ (here closure is taken with respect to $V_i^*$)
for each $i \in J$ and in particular $j \in I \setminus J$. For each~$j$ with
this property, define $X_j \coloneq \emptyset$ and $Y_j \coloneq V_j^*$. We
recall $C_j \subseteq Y_j$. As $I$ is a finite set, we can shrink each
$X_i, i \in J$ such that $X_i \cap \overline{V}_{i,j}^* = \emptyset$ holds
for each $j \in I$ with $x \not \in \overline{q(V_j^*)}$. \bigskip

  \textbf{Step 3:} Consider $j \in I$ with $x \in \partial q(V_j^*)
= \overline{q(V_j^*)} \setminus q(V_j^*)$, i.e.\ $q^{-1} (x) \cap V_j^*
= \emptyset$ is satisfied. Since $x \in \partial q(V_j^*)$, there is
$i \in J$ with $x_i \in \overline{V}_{i,j}^*$ (the sets $V_{i,j}^*$ are the
domains of the change of chart mappings
$\psi_{i,j}|_{V_{i,j}^*}^{V_{j,i}^*}$).\\ 
  Consider the closure $\overline{C}_j$ of $C_j \subseteq V_j^*$ with respect
to~$E_\C$. We prove that $\psi_{i,j} (x_i)$ is not contained in
$\overline{C}_j$. To reach a contradiction,
suppose this was wrong.
We could then find a net $(y_\sigma)_{\sigma \in \Sigma} \subseteq C_j$
with $y_\sigma \rightarrow \psi_{i,j} (x_i)$. Recall from the proof of
Proposition \eqref{prop: THMA1} that
$\overline{V}_{i,j}^* \subseteq U_{i,j}^*$ is satisfied and the mapping
$\psi_{i,j}|_{U_{i,j}^*}^{U_{j,i}^*}$ is a real analytic diffeomorphism.
Hence $\psi_{i,j} (X_i \cap U_{i,j}^*)$ is an open neighborhood of
$\psi_{i,j} (x_i)$ in $U_{j,i}^*$ which contains $\psi_{i,j} (x_i)$.
Thus the following holds: 
  \begin{displaymath}
   (y_\sigma)_{\sigma \in \Sigma} \subseteq C_j \cap
\psi_{i,j} (X_i \cap U_{i,j}^*) \subseteq V_j^* \cap
\psi_{i,j} (X_i \cap U_{i,j}^*) \subseteq V_j^* \cap
\psi_{i,j} (V_i^* \cap U_{i,j}^*) = V_{j,i}^*.
  \end{displaymath}
  Now $q (\psi_{j,i} (y_\beta)) = q(y_\beta) \in C$ yields
$\psi_{j,i} (y_\beta ) \in C_i$. Furthermore, we observe
$\psi_{j,i} (y_\beta ) \in \psi_{j,i} \psi_{i,j} (X_i \cap U_{i,j}^*)
\subseteq X_i$. But this contradicts the choice of $X_i$ for $i \in J$,
as $C_i \cap X_i = \emptyset$ (cf.\ Step 1).
  In conclusion, $\psi_{i,j} (x_i)$ is not contained in~$\overline{C}_j$.
By regularity of~$E_\C$, there are open disjoint neighborhoods~$X_j'$
of~$\psi_{i,j} (x_i)$ and~$Y_j'$ of~$\overline{C}_j$. Define
$X_j \coloneq V_j^* \cap X_j'$ and $Y_j \coloneq V_j^* \cap Y_j'$.\bigskip

  Shrinking the open sets obtained for $i \in J$, we may achieve the
  following: 
  \begin{enumerate}
   \item[(a)] $\psi_{i,j} (X_i) = X_j$ for all $i,j \in J$.
   \item[(b)] For each $j \in I \setminus J$, i.e.\ $x_i \not \in V_{i,j}^*$, the
   set $X_i \cap U_{i,j}$ satisfies $\psi_{i,j} (X_i \cap U_{i,j})
   \subseteq X_j$. Observe that by Step~2 this condition is trivially
   satisfied if $x \not \in \overline{q(V_{i,j}^*)}$.
  \end{enumerate}
  Define neighborhoods $X \coloneq \bigcup_{i \in J} q (X_i)$ and
$Y \coloneq \bigcup_{i \in I} q(Y_i)$. These sets are open in~$\Omega^*$
and~$X$ is a neighborhood of~$x$. Since $q^{-1} (C)
= \bigcup_{i \in I} C_i \subseteq \bigcup_{i \in I} Y_i$, the open set~$Y$
contains~$C$.\\ We claim that~$X$ and~$Y$ are disjoint. If this is true,
then~$\Omega^*$ is a regular topological space and the proof is complete. To prove
disjointness, consider $z \in X \cap Y$. Choose $i \in I$ such that
$z \in q(V_i^*)$. As $q|_{V_i^*}$ is injective, there is a unique
preimage $\tilde{z} \in q^{-1} (z) \cap V_i^*$. Since $z\in X\cap Y$,
we obtain: 
  \begin{displaymath}
   \tilde{z} \in q^{-1} (X) \cap q^{-1} (Y) \cap V_i^* \subseteq X_i \cap
   \bigcup_{i \in I}\psi_{j,i} (Y_j \cap V_{i,j}^*).
  \end{displaymath} 
  The inclusion above holds for $i \in I \setminus J$ by property~(b)
for the~$X_i$. By construction, the sets $X_i$ and
$Y_i = \psi_{i,i} (Y_i \cap V_{i,i}^*)$ are disjoint, whence there must be
$j \neq i$ with $\tilde{z} \in \psi_{j,i} (Y_j \cap V_{i,j}^*)$. For
such~$i$, we obtain $\psi_{i,j}(\tilde{z}) \in Y_{j}$. However,
$\psi_{i,j} (\tilde{z})$ is the uniquely determined preimage of~$z$ with
respect to $q|_{V_j^*}$. Hence $\psi_{i,j} (\tilde{z})$ is an element
of~$X_j$ as $z \in X$ holds. Again, this contradicts
$Y_j \cap X_j = \emptyset$. We conclude that there cannot be such an
index~$j$. Summing up, the set $X\cap Y$ must be empty. 
\end{proof}

Sticking to the general theme of constructing complexifications for
infinite-dimensional manifolds we now turn to complexifications
of locally convex vector bundles.
For finite-dimensional vector bundles these results seem to be
part of the folklore.
The notation of vector bundles we use follows \cite[Section 3]{hg2013b}.

\begin{defn}
 Let $(E,\pi,M)$ be a real analytic locally convex vector bundle. 
   We say that  $(E,\pi,M)$ \emph{admits a bundle complexification}, 
   if there exists a complex analytic bundle $(E^*, p, M^*)$ such that
   $M^*$ is a complexification of $M$, $E^*$ is a complexification of $E$
   and the bundle projection $p$ restricts to the bundle projection $\pi$ on $E \subseteq E^*$. 
\end{defn}

\begin{lem}\label{lem: Bcompl}
   Let $N$ be a real analytic manifold with a complexification $N_\C$. 
    \begin{enumerate}
     \item[{\normalfont (a)}] The bundle $(TN,\pi_{TN},N)$ admits a bundle complexification,
     it is given by the bundle $(T(N_\C), \pi_{T(N_\C)},N_\C)$.
     \item[{\normalfont (b)}] If $N_\C$ is paracompact and $N$ is modeled on a metrisable space,
   then $TN \subseteq T(N_\C)$ is a paracompact complexification which is unique in the sense of Theorem B (a). 
   Note that in this case $N$ is a paracompact manifold.
    \end{enumerate}
\end{lem}

\begin{proof} 
\begin{enumerate}
  \item[{\normalfont (a)}] Let $E$ be the model space of $N$ and
$\Psi \colon N_\C \supseteq U_\Psi \rightarrow V_\Psi \subseteq E_\C$ be an adapted chart of $N_\C$,
i.e.\ $\psi \coloneq \Psi|_{U_\psi \cap N}$ is a chart for $N$.
Then each $T\Psi$ is a bundle trivialisation for the complex analytic bundle which
restricts to the bundle trivialisations $T\psi$ of $TN \rightarrow N$.
Note that the inclusion $\iota E \rightarrow E_\C$ is real analytic. 
Hence for each adapted chart $T\Psi$ of $N_\C$ the map
$I_\Psi \colon \pi_{TN}^{-1}(X_\psi) \rightarrow \pi_{T(N_\C)}^{-1} (X_{\Psi}),
I_\psi \coloneq T\Psi^{-1} \circ \iota \circ T\psi$ is real analytic.
Fix a family $\fF$ of adapted charts for $N_\C$ whose domain covers $N$.
Then the family $(I_\Psi)_{\Psi \in \fF}$ glues to a real analytic map
$I \colon TN \rightarrow T(N_\C)$ which is fibre-wise the inclusion $\iota$.
Moreover it is clear from the construction that $I$ takes $TN$ diffeomorphically
to a real analytic submanifold of $T(N_\C)$.
We conclude that (up to identification) $T(N_\C)$ is a complexification of $TN$.
 \item[{\normalfont (b)}] Let now $N$ be modeled on a metrizable space and
$N_\C$ be paracompact. Hence $E_\C$ is also metrizable.
Now $N_\C$ is paracompact and locally metrizable and the fibre of $T(N_\C)$ is
metrizable, the paracompactness of $T(N_\C)$ follows from \cite[Proposition 29.7]{KaM}.
 
To prove the uniqueness of the germ in the sense of Theorem B\,(a)
it suffices to see that the image of $I$ is closed in $T(N_\C)$.
As $N_\C$ is paracompact and modeled on a metrizable space, $N_\C$ is metrizable. 
Hence Proposition \ref{specialize}\,(a) shows that there is a paracompact
open neighborhood $\Omega \subseteq N_\C$ of $N$ such that $N$ is a
closed subset of $\Omega$. 
Now assume that without loss of generality $X_\Psi \subseteq \Omega$ for all $\Psi \in \fF$.
By paracompactness of $\Omega$, we shrink the family $\fF$ such that the
family of open sets $(X_\Psi)_\fF$ becomes a locally finite cover of $N$ .
Following \cite[Lemma 5.1.6]{Engelking1989}, we can choose a locally finite
closed cover $\cC$ of $\Omega$ subordinate to the locally finite open $\set{X_\Psi}_{\Psi \in \fF} \cup \set{\Omega \setminus N}$.
Denote for all $\Psi \in \fF$ by $A_\Psi$ the element in $\cC$ subordinate to the set $X_\Psi$.
We note that $N \subseteq \bigcup_{\Psi \in \fF} A_\Psi$ holds.
Then 
 $
   \im \, I = \bigcup_{\Psi} (T\Psi^{-1} (\Psi (A_\Psi \cap N) \times (E \times \set{0}))
 $
is closed as a union of a locally finite family of closed sets thus completing the proof.
\end{enumerate}
\end{proof}

Let us complete the picture for bundle trivialisations with finite dimensional paracompact base.

\begin{prop}
 Let $(E,\pi,M)$ be a real analytic locally convex bundle with finite dimensional base $M$.
 If $M$ is paracompact, then $(E,\pi,M)$ admits a bundle complexification.
 Moreover, the bundle complexification is unique
 \end{prop}

\begin{proof}
 Let $F$ be the typical fibre of the bundle $(E,\pi,M)$. 
 We describe the bundle structure of $(E,\pi,M)$ via cocycles. 
 To this end fix an atlas $\fF$ of real analytic bundle trivialisations for $(E,\pi,M)$ indexed by a set $I$.
 For each pair $(i,j) \in I^2$ we obtain a cocycle $g_{ij} \colon X_{ij} \coloneq
 X_{\psi_i} \cap X_{\psi_j} \times F \rightarrow F$, i.e.\ a real analytic mapping which is linear in the second component.
 We will now extend the cocycles to complex analytic mappings whose domains glue to a complex analytic vector bundle.
 
 As $M$ is a finite dimensional and paracompact real analytic manifold, $M$ admits a complexification $M^*$.
 Following \cite{grauert}, we can realize $M^*$ as a Grauert tube, i.e.\ $M^* \subseteq T^*M$ is an open
 $\R$-balanced neighborhood of the zero-section.
 Thus there is a canonical scalar multiplication on $M^*$ such that $[-1,1] \cdot M^* \subseteq M^*$ holds.
 
 Each cocycle $g_{ij}, i,j\in I$ extends to a complex analytic map
 $\tilde{g}_{ij} \colon \tilde{W} \rightarrow F_\C$ where $\tilde{W} \subseteq M^* \times F_\C$ is an open neighborhood of $U_{ij}$.
 Shrinking $\tilde{W}$ we can assume that $\tilde{W} = \tilde{U}_{ij} \times \Omega \subseteq M^* \times F_\C$,
 where $\tilde{U}_{ij} \subseteq M^*$ is an $\R$-balanced neighborhood of $U_{ij}$ and $\Omega \subseteq F_\C$ is a zero-neighborhood.
 Now we shrink $\Omega$ such that it becomes a $\C$-balanced set, i.e.\ $\overline{B_1^\C} (0) \cdot \Omega = \Omega$. 
 We consider the restriction of $\tilde{g}_{ij}$ to $\tilde{U}_{ij} \times \Omega$ which we also denote by $\tilde{g}_{ij}$.
 Then the complex analytic maps 
 \begin{align*}
  \tilde{U}_{ij} \times \Omega \times \overline{B_1^\C (0)} &\rightarrow F_\C , (x,v,z) \mapsto \tilde{g}_{ij} (x,zv) \text{ and }\\ 
  \tilde{U}_{ij} \times \Omega \times \overline{B_1^\C (0)} &\rightarrow F_\C , (x,v,z) \mapsto z\tilde{g}_{ij} (x,v)
 \end{align*}
 coincide on the the total real subset $U_{i,j} \times (\Omega \cap F) \times [0,1]$ as $g_{ij}$ is (real)-linear. 
 The identity theorem for real analytic maps thus shows that both maps coincide and we obtain the formula
 $\tilde{g}_{ij} (x,zv) = z \tilde{g}_{ij} (x,v)$ for all $(x,v,z) \in \tilde{U}_{ij} \times \Omega \times B_1^\C (0)$.
 A similar argument shows that for $v,w \in \Omega$ with $v+ w \in \Omega$ the formula
 $\tilde{g}_{ij} (x,v+w) = \tilde{g}_{ij} (x,v) +\tilde{g}_{ij} (x,w)$ holds.
 Thus for each pair $(i,j) \in I^2$ we obtain a well defined complex analytic map 
  \begin{displaymath}
   \tilde{h}_{ij} \colon \tilde{U}_{ij} \times F_\C \rightarrow F_\C,
   \tilde{h}_{ij} (x,v) \coloneq \tilde{g}_{ij} (x,zv) \cdot \frac{1}{z} \text{ for } z \in \C^\times \text{ with } zv \in \Omega
  \end{displaymath}
 which is complex linear in the second component.
% ZITAT???.
 
 We now replace $M^*$ with the open subset $\bigcup_{i,j \in I} \tilde{U}_{ij}$.
 Then the $\tilde{h}_{ij}$ are cocycles for a complex analytic bundle $E^* \rightarrow M^*$ 
 To see this we have to establish the cocycle condition. Consider $i,j,k \in I$ with $\tilde{U}_{ij} \cap \tilde{U}_{jk} \neq \emptyset$.
 For each $z \in \tilde{U}_{ij} \cap \tilde{U}_{jk}$ we have $0\cdot z \in U_{ij} \cap U_{jk} \subseteq \tilde{U}_{ij} \cap \tilde{U}_{jk}$.
 From the cocycle condition for the maps $g_{ij}$ we derive
 $\tilde{h}_{ij} (0z,\tilde{h}_{jk} (0z ,v)) = g_{ij} (0z,g_{jk} (0z,v))
 = g_{ik} (0z,v) = \tilde{h}_{ik} (0z,v)$ for all $v \in \Omega \cap F$.
 Again the identity theorem for complex analytic maps yields
 $\tilde{h}_{ij} (z,\tilde{h}_{jk} (z ,v)) = \tilde{h}_{ik} (z,v)$ for all $z \in \tilde{U}_{ij} \cap \tilde{U}_{jk} \times \Omega$.
 We conclude that the cocycles $(\tilde{h}_{ij})_{i,j\in I}$ give rise to a complex analytic bundle
 $\pi^* \colon E^* \rightarrow M^*$ which is a bundle complexification of $(E,\pi,M)$.
 
 To see that this bundle complexification is unique, let $(E^\bullet, p, M^\bullet)$ be another bundle complexification.
 As the complexification of the finite-dimensional paracompact manifold $M$ is unique in the sense of Theorem B\,(a),
 without loss of generality we may assume $M^\bullet = M^*$.
 Let $f_{\alpha\beta} \colon X_{\alpha\beta} \times F_\C \rightarrow F_\C$ with $\alpha, \beta \in J$ be a family of cocycles
 which defines the bundle structure for $(E^\bullet, p M^*)$.
 Note that $X_{\alpha,\beta} \subseteq M^*$ for all $\alpha, \beta \in J$.
 Going to open subsets of $E^\bullet$ and $M^*$, we may assume that $X_{\alpha\beta}$ is also $\R$-balanced.
 In particular this implies for $i,j\in I$ and $\alpha, \beta \in J$ that the intersection $\tilde{U}_{ij} \cap X_{\alpha\beta}$ is $\R$-balanced.
 Hence if $\tilde{U}_{ij} \cap X_{\alpha\beta} \neq \emptyset$ the cocycles $f_{\alpha,\beta}$ and $\tilde{h}_{ij}$ 
 restrict on the real subset $X_{\alpha\beta} \cap U_{ij} \times F$ to cocycles of the bundle $E \rightarrow M$.
 Thus the identity theorem for complex analytic maps shows that the composition of
 $f_{\alpha\beta}$ and $\tilde{h}_{ij}$ is a cocycle for both complex bundles $E^\bullet \rightarrow M^*$ and $E^*\rightarrow M^*$.
 We conclude that there is an open neighborhood $W$ of $E$ in $E^\bullet$ together with a complex analytic diffeomorphism
 $\varphi \colon W \rightarrow \varphi (W) \subseteq E^*$ such that $\varphi (W)$ is an open neighborhood of $E$ in $E^*$ and $\varphi|_E = \id_E$. 
\end{proof}
%
%
%
%
%
%
%
%
%
%
%
%
%
%
%
%
% \section{Lie group structure on {\boldmath$\Germ_\C(K,H)$}}\label{sec-lie-cx}
\section{The compactly regular locally convex vector space {\boldmath$\Germ_\C(K,Z)$}}				\label{sec-germ-compactly regular}
% selbst definierte Befehle, die ich gerne per Copy&Paste auch wieder rausschmei???e, sollten Sie st???ren:
\newcommand{\func}[5]{#1 \colon #2 \longrightarrow #3 : #4 \mapsto #5}
\newcommand{\smfunc}[3]{#1 \colon #2 \longrightarrow #3}
\newcommand{\nnfunc}[4]{#1 \longrightarrow #2 : #3 \mapsto #4}
\newcommand{\seqn}[1]{\left(#1\right)_{n\in \N}}

\newcommand{\BHol}[2]{\text{BHol}\left(#1,#2\right)}
\newcommand{\GermC}[2]{\Germ_\C  \left(#1,#2\right)}
\newcommand{\oBallin}[3]{\mathrm{B}_{#1}^{#2}\left(#3 \right)}
\newcommand{\cBallin}[3]{\overline{\mathrm{B}}_{#1}^{#2}\left(#3 \right)}
\newcommand{\cl}[1]{\left[ #1 \right]}
\newcommand{\supnorm}[1]{\norm{#1}_\infty}
\newcommand{\opnorm}[1]{\ensuremath{  \norm{  #1 }_{\mathrm{op}} }\!  }
\newcommand{\Znorm}[1]{\norm{#1}_Z}
\newcommand{\abs}[1]{\left|	#1	\right|}

\renewcommand{\epsilon}{\varepsilon}

\newcommand{\h}{\mathfrak{h}}
\newcommand{\g}{\mathfrak{g}}
\newcommand{\br}[2]{\left[ #1 , #2 \right]}
\newcommand{\hnorm}[1]{\norm{#1}_\h}
\newcommand{\BoundOp}[1]{\mathcal{L}\left(#1\right)}

\newcommand{\Aut}{\mathrm{Aut}}
\newcommand{\GL}{\mathrm{GL}}
\newcommand{\AD}{\mathrm{AD}}
\newcommand{\EXP}{\mathrm{EXP}}
\newcommand{\cG}{\mathcal{G}}
\newcommand{\cN}{\mathcal{N}}
\newcommand{\gen}[1]{\langle #1 \rangle}

Let $M$ be a $\K$-analytic manifold modeled on a metrizable locally convex topological
$\K$-vector
space $X$ and let $K\subseteq M$ be a compact subset.
For a $\K$-analytic Banach-Lie group $H$, we consider all
$\K$-analytic functions $\smfunc{\gamma}{U_\gamma}{H}$, defined on an open neighborhood
$U_\gamma$ of $K$ in~$M$.
Identifying two such functions if they agree on a common smaller neighborhood,
we obtain the set $\Germ_\K(K,H)$ of equivalence classes $[\gamma]$
(the \emph{germs} around $K$).
With the operation $[\gamma][\eta]:=[\gamma\eta]$ inherited
by pointwise multiplication of representatives, the set $\Germ_\K(K,H)$
becomes a group. In the special case that $Z$
is a Banach space over $\K$,
we define $z[\gamma]:=[z\gamma]$ for $z\in \K$
and $[\gamma]\in \Germ_\K(K,Z)$.
In this way, $\Germ_\K(K,Z)$
becomes a vector space over $\K$.\\[2.3mm]
In this and the following section, we shall always assume
that $\K=\C$; all manifolds and vector spaces will be complex.
In Section \ref{sec-reg-cx}, we show that the group $\Germ_\C(K,H)$
carries a $C^0$-regular Lie group structure. In the present section, we consider the easier case of
germs with values in a Banach space instead of a Banach Lie group:

To this end, for the rest of this section, let $Z$ be a complex Banach space. Then, as just explained,
the set $\GermC{K}{Z}$ carries a natural vector space structure. Our first step is to define a suitable topology on it.
We start by observing that each function $\smfunc{\gamma}{U}{Z}$ defined on an open neighborhood
of $K$ maps $K$ to a compact, hence bounded subset of the space $Z$. Therefore $\gamma(K)$ is
contained in an open ball $\oBallin{R}{Z}{0}$ for an $R>0$.
The preimage $V:=\gamma^{-1}(\oBallin{R}{Z}{0})$ is now an open neighborhood of $K$
such that $\gamma|_V$ is bounded. This means that each germ $\cl{\gamma}\in\GermC{K}{Z}$
can be represented by a \emph{bounded} analytic function on a
suitable open neighborhood of $K$.

For each open neighborhood $U\subseteq X$, the space $\BHol{U}{Z}$ of bounded
$\C$-analytic functions from $U$ to $Z$ is a closed vector subspace of the Banach space
$(BC(U,Z),\norm{\cdot}_\infty)$ of bounded continuous maps endowed with the
sup-norm and hence, $\BHol{U}{Z}$ is a Banach space as well.

This allows us to endow the space $\GermC{K}{Z}$ with the locally convex inductive
limit topology with respect to all Banach spaces $\BHol{U}{Z}$, where $U$ ranges
through all open neighborhoods of $K$. 
Note that at this moment, we do not know if this topology is Hausdorff.
Before examining this space in greater detail, we state a topological lemma
which will be helpful later on:
\begin{lem}						\label{lem_compact_union_and_countable_basis}
 Let $K$ be a compact subset of a manifold $M$, modeled on a metrizable space $X$.
 \begin{itemize}
  \item [\rm(a)] The set  $K$ can be written as a finite union of $($not necessarily disjoint$)$
  compact sets $K_1,\ldots, K_m$ such that each $K_j$ is contained in some chart.
  \item [\rm(b)] There exists a sequence $U_1\supseteq U_2 \supseteq U_3 \supseteq \cdots$ of open subsets,
  forming a basis of neighborhoods of $K$ in $M$, i.e.~each open neighborhood of $K$ contains
  one of the sets $U_n$ as a subset. This sequence may be chosen in a way such that each connected
  component of each $U_n$ has a nonempty intersection with $K$.
 \end{itemize}
\end{lem}
\begin{proof}
 (a):\\
 The domains of charts for the manifold $M$ form an open cover of the compact set $K$.
Therefore, $K$ is contained in a \emph{finite} union of open chart domains $U_1,\ldots, U_m$.
We will only show the statement for the case $m=2$, the general case follows by an easy induction argument.
To this end, assume that $K\subseteq U_1\cup U_2$. The compact sets $A_1:=K\setminus U_2$ and
$A_2:=K\setminus U_1$ are disjoint in the Hausdorff space $M$ and thus have disjoint open neighborhoods
$W_1$ and $W_2$, respectively. Now, we have the decomposition
$K=K_1\cup K_2$ with $K_1:= K\setminus W_2$ and $K_2:=K\setminus W_1$.

 (b):\\
 By part (a), we can write the compact set $K$ as a finite union $K=\bigcup_{j=1}^m K_j$ of
compact sets each of which is contained in a chart domain $U_j$. For each $K_j$ which is a
compact subset of a metrizable open set $U_j$, we can find a sequence $\seqn{U_n^{(j)}}$,
which is a basis of open neighborhoods of $K_j$ using a metric on $U_j$. By taking the union
$\widetilde{U}_n:=\bigcup_{j=1}^m U_n^{(k)}$, we obtain a sequence of open neighborhoods of $K$.
 With this construction, it is not automatic that each connected component of each $\widetilde{U}_n$
has a nonempty intersection with $K$. However, by setting $U_n$ as the union of all connected components
of $\widetilde{U}_n$ intersecting $K$ nontrivially, we obtain a sequence $\seqn{U_n}$ having the desired properties.
\end{proof}
\noindent
From now on, we fix a basis of open neighborhoods $\seqn{U_n}$ of $K$ as in part(b) of the preceeding
lemma and  consider the Banach spaces $\BHol{U_n}{Z}$ with the corresponding bonding maps

\[
 \func{\iota_{m,n}}{ \BHol{U_m}{Z} }{ \BHol{U_n}{Z} }{\gamma}{\gamma|_{U_m}} \hbox{ for }m\leq n.
\]
These mappings are obviously continuous (with operator norm at most one) and since each $U_n$ meets the
compact set $K$, we obtain the injectivity of the bonding maps by the Identity Theorem (Lemma \ref{idthm}).

Since the sequence $\seqn{U_n}$ is cofinal in the directed set of all open neighborhoods, we obtain the following
\begin{prop}							\label{prop_GermKZ}
The locally convex vector topology on the space $\GermC{K}{Z}$ defined above
makes it the locally convex direct limit of the sequence
\[
\bHol(U_1,Z)
 \mathop{\longrightarrow}^{\iota_{1,2}}
 \bHol(U_2,Z) \mathop{\longrightarrow}^{\iota_{2,3}}
 \bHol(U_3,Z) \mathop{\longrightarrow}^{\iota_{3,4}} \cdots
\]
In particular, $\GermC{K}{Z}$ is an \emph{(LB)}-space, i.e.~a direct limit of an
ascending \emph{sequence} of Banach spaces with injective continuous bonding maps.
The topology on $\GermC{K}{Z}$ does not depend on the choice of the sequence $\seqn{U_n}$.
\end{prop}
In general, (LB)-spaces need not be well-behaved, for example an (LB)-space need not be Hausdorff
and even if it is Hausdorff, it need not be complete. However, for the class of
\emph{compact regular} (LB)-spaces, the situation is nicer:
\begin{lem}[Criterion for compact regularity]							\label{lem_compact_regularity}
 Let $E:= \bigcup_{n=1}^\infty E_n$ be a locally convex direct limit of Banach spaces.
 Consider the following statements:
 \begin{itemize}
  \item [\rm(i)]	For every $n\in\N$, there is an $m\geq n$ such that
  for each $\ell\geq m$, there is an open absolutely convex $0$-neighborhood $\Omega$ of $E_n$
  such that $E_\ell$ and $E_m$ induce the same topology on the set $\Omega$.
  \item [\rm(ii)]
		$ \forall n\in\N \ \exists m\geq n\ \forall \epsilon>0 , \ell\geq n\ \exists \delta>0 :
  \cBallin{1}{E_n}{0} \cap \cBallin{\delta}{E_\ell}{0}\subseteq \cBallin{\epsilon}{E_m}{0}$.
  \item [\rm(iii)]	The sequence $\seqn{E_n}$ is \emph{compactly regular}, i.e.~for every compact
  subset $C$ of $E$ there is an index $n\in\N$ such that $C$ is compact in $E_n$.
  \item [\rm(iv)]	The locally convex vector space $E$ is Hausdorff and complete.
 \end{itemize}
 Then \emph{(i), (ii)} and \emph{(iii)} are equivalent  and imply \emph{(iv)}.
\end{lem}
\begin{proof}
The equivalence of (i) and (ii) is clear, since (ii) is just a restatement of (i). The rest of this proposition
follows from statements in \cite{wengenroth}:
To be more concrete: 
\cite[Theorem 6.4]{wengenroth} says that (i) is equivalent to (iii) and that both are equivalent to a property
of $\seqn{E_n}$ called \emph{acyclicity}. \cite[Proposition 6.3]{wengenroth} shows that the limit topology
of an acyclic sequence is Hausdorff and  \cite[Corollary 6.5]{wengenroth} tells us that the limit
of an acyclic sequence is complete. All of this is true in the more general setting of (LF)-spaces
but we shall not use this here.
\end{proof}
\noindent
In the rest of this section, we show that $\GermC{K}{Z}$ is compactly regular.
This will not only garantee that the space is a complete Hausdorff topological vector space but also
that the (local and global) Lie groups we will construct later in Theorem \ref{thm_local_regularity}
and Theorem \ref{thm_GermC_Lie_group_structure} are $C^0$-regular.

As a first step, we consider the case where there is no manifold $M$ involved, i.e.~$K$ is a compact subset of the space $X$.
The following fact will be useful:

\begin{lem}[Factorization of bounded holomorphic functions]			\label{lem_Banach_factor}
 Let $X$ be a complex locally convex space and let $\smfunc{p}{X}{[0,+\infty[}$ be a continuous
seminorm with open unit ball $B^p\subseteq X$. Let $X_p$ be the Banach space obtained as the
completion of the normed space $(X/p^{-1}(\set{0}),p)$
and let $\smfunc{\pi_p}{X}{X_p}$ denote the canonical continuous linear map.
For a subset $K\subseteq X$, we consider the open set $U:=K+ B^p$.
Then the following map is an isometric isomorphism of Banach spaces:
 \[
  \nnfunc{\bHol(W,Z) }{\bHol(U,Z)}{\gamma}{\gamma\circ \pi_p,}
 \]
 where $W:=\pi_p(K)+ \oBallin{1}{X_p}{0}$ is the corresponding open subset in the Banach space $X_p$.

 In particular, every bounded complex analytic map defined on $U$ factors through the Banach space $X_p$.
\end{lem}
\begin{proof}
 It is clear that $\pi_p(U)$ is dense in $W$. This implies that the map
 \[
  \nnfunc{\BHol{W}{Z} }{\BHol{U}{Z}}{f}{f\circ \pi_p,}
 \]
 is an isometric embedding. It remains to show the surjectivity. To this end,
let $\eta\in\BHol{U}{Z}$ be given. We will show that there is a $\gamma\in\BHol{W}{Z}$ such that $\eta=\gamma\circ\pi_p$.

 Let $a\in K$ be a point. Then $\eta$ admits a power series expansion around $a$, i.e.~for $x\in a+B^p$, we have:
 \[
  \eta(x)=\sum_{k=0}^\infty \beta_k(x-a,\ldots,x-a)
 \]
 with continuous symmetric $k$-linear maps $\smfunc{\beta_k}{X^k}{Z}$.

 Let $v\in X$ be a fixed vector and set $R:=\frac{1}{p(v)}\in\, ]0,+\infty]$. Then we define the following
function of one complex scalar variable:
 \[
  \func{h}{\oBallin{R}{\C}{0}}{Z}{z}{f(a+zv) = \sum_{k=0}^\infty \beta_k(v,\ldots,v) \cdot z^k  .}
 \]
 The coefficients of this series can be computed via Cauchy's integral formula:
 \[
  \beta_k(v,\ldots,v)= \frac{1}{2\pi i}\int_{\abs{z}=r} \frac{h(z)}{z^{k+1}} dz =  \frac{1}{2\pi i}\int_{\abs{z}=r} \frac{f(a+zv)}{z^{k+1}} dz.
 \]
 Applying the norm on both sides, we get the estimate:
\begin{eqnarray*}
\Znorm{\beta_k(v,\ldots,v)} &=& \left\| \frac{1}{2\pi i}\int_{\abs{z}=r} \frac{f(a+zv)}{z^{k+1}} dz \right\|_Z
\leq \frac{1}{2\pi}\cdot \frac{2\pi r \supnorm{f}}{r^{k+1}}\\
&=& \supnorm{f}\cdot \frac{1}{r^k}\leq \supnorm{f} (p(v))^k.
\end{eqnarray*}
With use of the polarization formula (e.g.~\cite[Theorem A]{MR0313810}), we obtain that each
$k$-linear map $\smfunc{\beta_k}{X^k}{Z} $ is continuous with respect to the seminorm $p$. This
implies that it factors
through the normed space $X/{p^{-1}(0)}$ to a continuous $k$-linear map and using the completeness
of the range space $Z$, we obain a continuous extension $\widetilde{\beta_k}$ to the Banach space $X_p$.

The power series $\sum_{k=0}^\infty \widetilde{\beta_k}$ so obtained converges on
the open ball $\oBallin{1}{X_p}{\pi_p(a)}$ to a complex analytic $Z$-valued function.

Now, we let the point $a\in K$ vary and obtain a complex analytic map on each
ball $\oBallin{1}{X_p}{\pi_p(a)}$. By construction, it is clear that the functions so obtained
agree on intersecting balls. Glueing together these functions, we get the function
 $   \smfunc{\gamma}{W}{Z} $
with the desired properties.
\end{proof}
The preceding lemma enables us to restrict our attention to the case where the domain
is a Banach space. This is useful because for functions defined on Banach spaces,
we have the following tool which can be found in \cite[Lemma 1.5]{dahmen2010}:
\begin{lem}[Absolute convergence of families of bounded power series]	\label{lem_absolute_family}
Let $K\subseteq X$ be a nonempty subset of a complex normed vector space $X$.
Let $W:=K+\oBallin{R}{X}{0}= \bigcup_{a\in K}\oBallin{R}{X}{a}$ be a union of
open balls with fixed radius \hbox{$R>0$}. Now, consider a set $M$ of bounded
complex analytic mappings from $W$ to a normed space $Z$ such that
$\sup_{\gamma\in M} \supnorm{\gamma}<\infty$.
Then we have for all $r<\frac{R}{2e}$ the following estimate:
\[
\sum_{k=0}^\infty \sup_{ \substack{\gamma\in M\\a\in K} } \frac{\opnorm{\gamma^{(k)}(a)}}{k!} r^k
 \leq \frac{R}{R-2er} \cdot \sup_{\gamma\in M} \supnorm{\gamma}.
\]
\end{lem}

\begin{lem}[Compact Regularity in the case that $M=X$]		\label{prop_GermCKZ_compact_regularity_special_case_X}
 Let $X$ be a metrizable complex locally convex vector space and let $K\subseteq X$
be a nonempty compact subset. Let $Z$ be a complex Banach space. Then the locally
convex direct limit
 \[
  \GermC{K}{Z}=\bigcup_{n\in\N} \bHol(U_n,Z)
 \]
is Hausdorff and compactly regular. Here $\seqn{U_n}$ is as in
Lemma \emph{\ref{lem_compact_union_and_countable_basis}\,(b)}.
\end{lem}
\begin{proof}
 For the whole proof, we fix a real number $r>0$ which is strictly less
than $\frac{1}{2e}$, where $e$ denotes Euler's number.

Since the space $X$ is metrizable and locally convex, its topology is
generated by a sequence of seminorms $\seqn{p_n}$. Replacing each $p_n$ by the seminorm
\[
 \widetilde{p_n}:= \frac{1}{r^n} \sum_{k\leq n} p_k,
\]
we obtain a new sequence of seminorms, still generating the same topology, but with the
additional property that $\widetilde{p_n}\leq r \cdot \widetilde{p_{n+1}}$. To simplify notation,
we call this new sequence of seminorms again $\seqn{p_n}$.
 Let $\oBallin{\epsilon}{p_k}{0}$ be an open ball around zero with respect to the seminorm $p_k$.
Since $r<1$,  there is a number $n\geq k$ such that $r^{n-k}<\epsilon$ and therefore,
the given ball $\oBallin{\epsilon}{p_k}{0}$ contains the unit ball $B^{p_n}= \oBallin{1}{p_n}{0}$.
This shows that the sequence of unit balls $\seqn{B^{p_n}}$ is a basis of $0$-neighborhoods in $X$.

 From now on, we will use this basis of $0$-neighborhoods to construct a basis of
open neighborhoods of the compact set $K$ by setting
 \[
  U_n:= K + B^{p_n}.
 \]
 By construction, each component of $U_n$ intersects $K$ non-trivially,
so $\seqn{U_n}$ is a sequence of the form in Lemma \ref{lem_compact_union_and_countable_basis}(b).

 In analogy to the notation in Lemma \ref{lem_Banach_factor}, we let $X_n$ denote the Banach space
obtained by completing the normed space $(X/p_n^{-1}(0),p_n)$.
 The unit ball in $X_n$ is denoted by $B^n$, the canoncal map by $\smfunc{\pi_n}{X}{X_n}$
and the bonding maps by $\smfunc{\eta_{k_1,k_2}}{X_{k_2}}{X_{k_1}}$:
 \[
  \xymatrix{ 			& &	X \ar[lld]_{\pi_1} \ar[d]^{\pi_2} \ar[drr]^{\pi_3} \ar[drrrr]^{\pi_4}	& &				& &				\\
		      X_1	& &	X_2 \ar[ll]_{\eta_{1,2}}					& &	X_3 \ar[ll]_{\eta_{2,3}}	& & \cdots \ar[ll]_{\eta_{3,4}} }
 \]
 By construction of the spaces $X_n$, it is clear that the maps $\smfunc{\eta_{n,n+1}}{X_{n+1}}{X_n}$
are continuous linear with operator norm $\opnorm{\eta_{n,n+1}}\leq r$ and with dense image.
 
 Let $K_n:=\pi_n(K)$ be the compact subset of $X_n$, and $W_n:=K_n+B^n$ be the corresponding open
neighborhood of $K_n$ in the Banach space $X_n$.

 We want to show that the direct limit
 \[
  \BHol{U_1}{Z} \to \BHol{U_2}{Z} \to \cdots
 \]
 is compactly regular. By Lemma \ref{lem_Banach_factor}, each $\gamma\in\BHol{U_n}{Z}$ factors through $\pi_n$
and we may identify the Banach spaces $\BHol{U_n}{Z} \cong \BHol{W_n}{Z}$.
It remains to show that the direct limit
 \[
  E_1:=\BHol{W_1}{Z} \to E_2:=\BHol{W_2}{Z} \to \cdots
 \]
 is compactly regular. We denote the bonding maps by
 \[
  \func{\iota_{k_1,k_2} }{E_{k_1}}{E_{k_2}}{\gamma}{\gamma\circ\eta_{k_1,k_2}}.
 \]
 We will use (ii) in the characterisation in Lemma \ref{lem_compact_regularity}, i.e.~let $n\in\N$ be given and set $m:=n+1$.
Let $\ell\geq n+1$ and $\epsilon>0$ be given. It remains to show that there is a number $\delta>0$ such that
 \[
  \cBallin{1}{E_n}{0} \cap \cBallin{\delta}{E_\ell}{0}\subseteq \cBallin{\epsilon}{E_m}{0}. 
 \]
 We may apply Lemma \ref{lem_absolute_family} to the subset $K_n$ of the Banach space $X_n$,
the set $M:=\cBallin{1}{E_n}{0}$, the number $R:=1$ and $r>0$ as already defined at the beginning of this proof.
 Then we may conclude the convergence of the series
 \[
  \sum_{k=0}^\infty s_k r^k \quad\hbox{ with }\quad s_k:=\sup_{\substack{\gamma\in \cBallin{1}{E_n}{0} \\ a\in K_n}} \frac{\opnorm{\gamma^{(k)} (a)  }}{k!} .
 \]
 Since this sum is convergent, there is an index $k_0\in\N$ such that
 $
  \sum_{k>k_0} s_k r^k \leq \frac{\epsilon}{2}.
 $
 Now, we are able to define the desired number $\delta$ as:
 \[
  \delta := (1-2er) r^{k_0}\cdot \frac{\epsilon}{2}.
 \]
 It remains to show that $   \cBallin{1}{E_n}{0} \cap \cBallin{\delta}{E_\ell}{0}\subseteq \cBallin{\epsilon}{E_m}{0}$.
 To this end, let $\cl{\gamma}\in \cBallin{1}{E_n}{0} \cap \cBallin{\delta}{E_\ell}{0}$ be given.
 Since $\cl{\gamma}\in \cBallin{1}{E_n}{0}$, we can view $\cl{\gamma}$ as an element
$\gamma_n\in E_n=\BHol{W_n}{Z}$ with $\norm{\gamma_n}_{E_n}\leq 1$.
 By $\cl{\gamma}\in\cBallin{\delta}{E_\ell}{0}$, we know that $\gamma_\ell:=\iota_{n,\ell}(\gamma_n)$
is an analytic function on $W_\ell$ and $\norm{\gamma_\ell}_{E_\ell}<\delta$. We want to show that
$\norm{\gamma_{n+1}}_{E_{n+1}}=\sup_{x\in W_{n+1}}\norm{\gamma_{n+1}(x) }_Z  \leq\epsilon$. To this end,
let $x_{n+1}\in W_{n+1}$ be given. It remains to show that $\norm{\gamma_{n+1}(x_{n+1})}_Z\leq \epsilon$.

 Since $\eta_{n+1,\ell}(W_{\ell})$ is dense in $W_{n+1}$, we may assume that
$x_{n+1}=\eta_{n+1,\ell}(x_\ell)$ for an $x_\ell\in W_\ell=K_\ell+B^\ell$, i.e.~the element $x_\ell$ can be written as
 \[
  x_\ell = a_\ell + v_\ell \hbox{ with } a_\ell\in K_\ell \hbox{ and } \norm{v_\ell}_{X_\ell}<1.
 \]
 Now, we can estimate the value of $\gamma_{n+1}(x_{n+1})$:
 \begin{align*}
  \Znorm{\gamma_{n+1}(x_{n+1})}	  &	=	\Znorm{\gamma_n(x_n)}
				\\&	\leq	\sum_{k\leq k_0} \Znorm{	\frac{\gamma_n^{(k)}(a_n)}{k!}(v_n,\ldots, v_n) 	}
					      +	\sum_{k >  k_0} \Znorm{	\frac{\gamma_n^{(k)}(a_n)}{k!}(v_n,\ldots, v_n) 	}	.			\tag{$*$} \label{eqn_two_sums}
 \end{align*}
 The first sum of (\ref{eqn_two_sums}) can be estimated by:
 \begin{align*}
  \sum_{k\leq k_0} \Znorm{	\frac{\gamma_n^{(k)}(a_n)}{k!}(v_n,\ldots, v_n) 	}
				  &	=	\sum_{k\leq k_0} \Znorm{	\frac{\gamma_\ell^{(k)}(a_\ell)}{k!}(v_\ell,\ldots, v_\ell) 	}
				\\&	\leq	\sum_{k\leq k_0} \frac{\opnorm{\gamma_\ell^{(k)}(a_\ell) }}{k!}\underbrace{\norm{v_\ell}_{E_\ell}^k}_{\leq1} \cdot r^{k}\cdot r^{-k_0}
				\\&	\leq	\left(\sum_{k=0}^\infty  \frac{\opnorm{\gamma_\ell^{(k)}(a_\ell) }}{k!} \cdot r^{k} \right)\cdot r^{-k_0}
 \end{align*}
 If we apply Lemma \ref{lem_absolute_family} to the Banach space $X_{\ell}$, the one element family $M=\set{\gamma_\ell}$
and parameters $R=1$ and $r>0$ as above, we obtain that this sum can be bounded above by
 \[
  \frac{1}{1-2er}\underbrace{\supnorm{\gamma_\ell}}_{\leq\delta} \cdot r^{-k_0} = \frac{1}{1-2er}\cdot \delta \cdot r^{-k_0} = \frac{\epsilon}{2}
 \]
 by the choice of $\delta$.
 
 The second sum of (\ref{eqn_two_sums}) is equal to:
 \begin{align*}
  \sum_{k >  k_0} \Znorm{	\frac{\gamma_n^{(k)}(a_n)}{k!}(v_n,\ldots, v_n) 	}
				  &	\leq	\sum_{k>k_0} \underbrace{  \frac{\opnorm{\gamma_n^{(k)} (a_n) }}{k!}  }_{\leq s_k} \left(\norm{\eta_{n,n+1} (v_{n+1}) }_{X_n}\right)^k
				\\&	\leq	\sum_{k>k_0} s_k  \Bigl(
									\underbrace{	 
										      \opnorm{\eta_{n,n+1}}	
										  }_{\leq r}
									\cdot \underbrace{
											   \norm{  v_{n+1}    }_{X_{n+1}}
											 }_{<1}
								  \Bigr)^k
				\\&	\leq	\sum_{k>k_0} s_k r^k
				   	\leq	\frac{\epsilon}{2}.
 \end{align*}
 This finishes the proof.
\end{proof}

 Let us now return to the general case of $K$ being a compact subset of a manifold $M$ which is
modeled on the metrizable locally convex space $X$. We will use the following easy tools
from the theory of locally convex direct limits:
\begin{lem}[Products and subspaces of compact regular direct limits]	\label{lem_compact_regularity_products_and_subspaces}
$\;$\\[-4mm]
 \begin{itemize}
  \item [\rm(a)]
  Let $E=\bigcup_n E_n$ and $F=\bigcup_n F_n$ be compactly regular direct limits
  of Banach spaces $\seqn{E_n}$ and $\seqn{F_n}$ respectively. Then the product $E\times F$ can
  be written as the compactly regular direct limit
		\[
		  E\times F = \bigcup_n (E_n \times F_n).
		\]
  \item [\rm(b)] Consider two sequences of locally convex spaces $\seqn{E_n}$ and $\seqn{F_n}$ respectively,
  together with topological embeddings $\smfunc{\tau_n}{E_n}{F_n}$ such that the following diagram commutes:
                \[
                 \xymatrix{
			      E_1 \ar[r] \ar@{^{(}->}[d]^{\tau_1} 	& E_2 \ar[r]\ar@{^{(}->}[d]^{\tau_2}	& E_3 \ar@{^{(}->}[d]^{\tau_3} \ar[r]	&	\cdots	\\
			      F_1 \ar[r]        	 		& F_2 \ar[r]       			& F_3 \ar[r]  				&	\cdots	
			  }
                \]
	      Assume that $E_m\cap \tau^{-1}_n(F_n) \subseteq E_n$ for all $m,n\in \N$ and that the sequence
              $\seqn{F_n}$ is compactly regular. Then $\seqn{E_n}$ is compactly regular as well.
	      $($Note that we do not state that $\bigcup_n E_n$ is a topological subspace of $\bigcup_n F_n$
              as this will not be true in general.$)$
 \end{itemize}
\end{lem}
\begin{proof}
 (a):\\
 Since the product of two locally convex vector spaces agrees with the direct
sum whose topology is a final locally convex topology
(as are locally convex direct limit topologies),
we may use the well-known transitivity of
final locally convex topologies and
obtain that the product topology on $E\times F$ agrees
with the locally convex direct limit topology:
 \[
  E\times F = \bigcup_n (E_n \times F_n).
 \]
% ZITAT Tatsuuma et al einfuegen
%
Let $C\subseteq E\times F$ be a compact set. Then the projections on the two factors $E$ and $F$
yield compact subsets $C_E$ and $C_F$ respectively. By compact regularity of $E$ and $F$, we find
a number $n\in\N$ such that $C_E$ is a compact subset of $E_n$ and $C_F$ is a compact subset of $F_n$.
This implies that the product $C_E\times C_F$ is a compact subset of $E_n\times F_n$, and
since $C\subseteq C_E\times C_F$, the claim follows.

 (b):\\
 Since every $\smfunc{\tau_n}{E_n}{F_n}$ is continuous linear, we obtain a continuous
linear map $\smfunc{\tau}{E}{F}$ by the universal property of the direct limit.
The compact set $C\subseteq E$ is then mapped onto the compact set $\tau(C)\subseteq F$
which is a compact subset of $F_n$ for a number $n\in\N$ by the compact
regularity of $\seqn{F_n}$. Let $x\in C\subseteq E$ be given. Then there exists a
number $m\in\N$ such that $x\in E_m$. By the hypothesis, this implies that $x\in E_n$.
Since $x\in C$ was arbitrary, this implies that $C$ is a subset of $E_n$. Using that
$\smfunc{\tau_n}{E_n}{F_n}$ is a topological embedding, $C$ is compact in $E_n$
since $\tau_n(C)$ is compact in $F_n$.
\end{proof}
Now, we are able to show the main result of this section:
\begin{prop}[Compact Regularity of $\GermC{K}{Z}$]		\label{prop_GermCKZ_compact_regularity}
 Let $M$ be a manifold modeled on the metrizable locally convex space $X$. Let $Z$ be a complex Banach space.
 Then for each compact subset $K\subseteq M$ the space $\GermC{K}{Z}$ is the compact regular direct limit
of the spaces $(\bHol(U_n,Z))_{n\in \N}$, where $\seqn{U_n}$ is any sequence of open neighborhoods
as in part \emph{(b)} of Lemma \emph{\ref{lem_compact_union_and_countable_basis}}.
 In particular, the locally convex space $\GermC{K}{Z}$ is Hausdorff and complete.
\end{prop}
\begin{proof}
 In the case that $K$ is so small that it is contained in one chart, we may assume
that $K\subseteq X$ and hence $\GermC{K}{Z}$ is compactly regular by
Proposition \ref{prop_GermCKZ_compact_regularity_special_case_X}. Consider the following

 Claim: Let $K',K''\subseteq M$ be compact subsets such that $\GermC{K'}{Z}$ and $\GermC{K''}{Z}$
are compactly regular. Then $\GermC{K'\cup K''}{Z}$ is compactly regular as well.

 If this claim is true, then the proposition follows, as each compact set $K$ can be written
as a finite union of compact sets each of which is contained in one chart by part (a)
of Lemma \ref{lem_compact_union_and_countable_basis}.

 It remains to show the claim:
 To this end, let $K'$ and $K''$ be two compact subsets with compactly regular direct limits
$\GermC{K'}{Z} $ and $\GermC{K''}{Z}$, respectively. By part (b) of
Lemma \ref{lem_compact_union_and_countable_basis}, we obtain open neighborhoods
$\seqn{U'_n}$ of $K_1$ and $\seqn{U''_n}$ of $K_2$ respectively. Let us denote
$F_n':= \BHol{U_n'}{Z}$ and $F_n'':=\BHol{U_n''}{Z}$ and $E_n:=\BHol{U_n'\cup U_n''}{Z}$.
The space $E_n$ can be embedded in the product $F_n'\times F_n''$ via
 \[
  \func{\tau_n}{E_n}{F_n'\times F_n''}{\gamma}{\left( \gamma|_{U_n'} , \gamma|_{U_n''} \right)}.
 \]
The direct limit $ \bigcup_n \left( F_n'\times F_n'' \right)$ is compactly regular by
\ref{lem_compact_regularity_products_and_subspaces}\,(a).
 There are two cases to consider: If $K'\cap K''=\emptyset$, then we may assume
that for each $n\in\N$, the two open sets, $U_n'$ and $U_n''$, are disjoint.
Therefore the mappings $\tau_n$ defined above are isomorphisms.
This shows that $\bigcup_n E_n\cong \bigcup_n (F_n'\times F_n'')$ is compactly regular.

Let us assume now that $K'\cap K''\neq \emptyset$. Let $m,n\in\N$ be given.
If we are able to show that $E_m\cap \tau^{-1}(F_n)\subseteq E_n$, the compact regularity
of $\bigcup_n E_n$ follows by part (b) of Lemma \ref{lem_compact_regularity_products_and_subspaces}.
 To this end, let $\gamma\in E_m$, i.e.~a function $\smfunc{\gamma}{U_m'\cup U_m''}{Z}$
be given and assume that $\tau(\gamma) = (\gamma|_{U'},\gamma|_{U''})\in\left( F_n'\times F_n'' \right)$.
This means that the function $\gamma$ can be extended holomorphically to a function
$\smfunc{\gamma'}{U_n'}{Z}$ and to a function $\smfunc{\gamma''}{U_n''}{Z}$.
The domains $U_n'$ and $U_n''$ are not disjoint since they both contain the nonempty
set $K'\cap K''$. Hence, by the Identity Theorem (Lemma \ref{idthm}), the two
extensions $\smfunc{\gamma'}{U_n'}{Z}$ and $\smfunc{\gamma''}{U_n''}{Z}$ can be combined
to obtain an extension on the union $U_n'\cup U_n''$. This shows that $\gamma\in E_n$ and finishes the proof.
\end{proof}
\section{Construction of a regular Lie group structure on {\boldmath$\Germ_\C(K,H)$}}\label{sec-reg-cx}
In this section, we show Theorem~C stated in the introduction, namely that there
exists a $C^0$-regular Lie group structure on the space of Lie group valued germs.
Before we consider global Lie groups, we first recall the notion of a \emph{local Lie group}:
\begin{defn}[Local Lie group]						\label{defn_local_lie_group}
 \begin{itemize}
  \item [(a)] 
	Let $G$ be a smooth manifold, $D\subseteq G\times G$ an open subset, $1\in G$, and let
	$
	 \func{m_G}{D}{G}{(x,y)}{x*y},
	$
	$
 	\func{\eta_G}{G}{G}{x}{x^{-1}}
	$
	be smooth maps. We call $(G , D , m_G , 1_G , \eta_G)$ a (smooth) \emph{local Lie group} if
	\begin{itemize}
 	 \item [(Loc1)] Assume that $(x,y),(y,z)\in D$. If $(x*y,z)$ or $(x,y*z)\in D$,
         then both are contained in $D$ and $(x*y)*z=x*(y*z)$.
 	 \item [(Loc2)] For each $x\in G$ we have $(x,1_G),(1_G,x)\in D$ and $x*1_G=1_G*x=x$.
 	 \item [(Loc3)] For each $x\in G$ we have $(x,x^{-1}),(x^{-1},x)\in D$ and $x*x^{-1}=x^{-1}*x=1_G$.
 	 \item [(Loc4)] If $(x,y)\in D$, then $(y^{-1},x^{-1})\in D$.
	\end{itemize}
 \item [(b)] The \emph{Lie algebra} $L(G):=T_{1_G}G$ of a local Lie group $G=(G,D,m_G,1_G,\eta_G)$
        is the tangent space of the manifold $G$ at point $1_G$. As in the case of (global) Lie groups,
        there is a natural structure of a locally convex Lie algebra on $L(G)$. 
 \item [(c)] A local Lie group~$G$ is called \emph{$C^0$-regular} if there is an open $0$-neighborhood
        $\Omega\subseteq C([0,1],L(G))$ such that each continuous curve $\gamma\in \Omega$ admits a
        continuously differentiable left evolution $\eta=\eta_\gamma\colon [0,1]\to G$ determined by
	\[
	  \eta(0)=1\quad\mbox{and}\quad (\forall t\in [0,1])\;\; \eta'(t)=\eta(t).\gamma(t),
	\]
	and the evolution map $\evol_G\colon \Omega\to G$, $\evol(\gamma):=\eta_\gamma(1)$
        is smooth (i.e., $C^\infty$).
\end{itemize}
Real analytic and complex local Lie groups are defined analogously.
In these cases, smoothness is replaced with real and complex analyticity,
respectively.
\end{defn}
\noindent
See \cite[Definition II.1.10]{Neeb2006} for more details on local Lie groups.
The following fact (see \cite{bourbaki1998}) gives a description
of local Lie groups modeled on Banach spaces:
\begin{numba}[Local Banach Lie groups]					\label{prop_local_Banach}
Let $\h$ be a complex Banach Lie algebra. Then there exits an absolutely convex open
$0$-neighborhood $W$ such that the Baker-Campbell-Hausdorff-series (BCH-series for short)
converges to a complex analytic map $\smfunc{*}{W\times W}{\h}$. By setting
$D:=\setm{(x,y)\in W\times W}{x*y\in W}$, we obtain a local Banach Lie group $(W, D, *, 0, -\id_W)$.
 Furthermore, \emph{every} local Banach Lie group is locally isomorphic to one which
is obtained in this fashion.
\end{numba}
\noindent
Now, we return to the setting of Section \ref{sec-germ-compactly regular}:
Let $M$ be a complex analytic manifold modeled on a complex metrizable locally
convex space $X$ and let $K$ be a compact non-empty subset.
We fix a basis of open neighborhoods $\seqn{U_n}$ as in
Lemma \ref{lem_compact_union_and_countable_basis}\,(b).
As in the last section, all vector spaces will be complex.
\begin{thm}[Regularity of the local group of germs]	\label{thm_local_regularity}
 Let $\h$ be a complex Banach-Lie algebra.
%  i.e.~ a Banach space $\h$ with a Lie bracket
%  $\smfunc{\br{\cdot}{\cdot}}{\h\times \h}{\h}$ such that 
%  \[
%   \hnorm{\br{x}{y}} \leq \hnorm{x} \hnorm{y} \hbox{  for all } x,y\in \h.
%  \]
The pointwise Lie bracket of $\h$-valued germs in $\GermC{K}{\h}$ is continuous
and turns $\GermC{K}{\h}$ into a complete locally convex topological Lie algebra.

Furthermore, there is an open $0$-neighborhood $\Omega\subseteq \GermC{K}{\h}$
such that the BCH-series converges to a complex analytic map $\smfunc{*}{\Omega\times \Omega}{\GermC{K}{\h}}$.
The open set $\Omega$ together with this multiplication becomes a complex \emph{local} Lie group which is $C^0$-regular.
\end{thm}

\begin{proof}[Proof of Theorem \ref{thm_local_regularity}]
For each $n\in\N$ the Banach space $\g_n:= \BHol{U_n}{\h}$ admits a pointwise Lie bracket and becomes
a Banach-Lie algebra in its own right with a corresponding BCH-multiplication which corresponds to
the pointwise BCH-multiplication of analytic maps.

By \cite[Theorem 4.5 (a)]{Dahmen_UnionGroupRegular_JLT}, we conclude that the locally
convex direct limit $\GermC{K}{\h}$ becomes a locally convex topological Lie algebra
(the completeness of the space of germs was already shown in Proposition \ref{prop_GermCKZ_compact_regularity})
admitting an open zero-neighborhood such that the BCH-multiplication defines a
local Lie group structure. Furthermore, since the sequence $\seqn{\g_n}$ is compactly regular
by Proposition \ref{prop_GermCKZ_compact_regularity}, we may apply part (c) of
\cite[Theorem 4.5]{Dahmen_UnionGroupRegular_JLT} to conclude that
the local Lie group so obtained is $C^0$-regular
(which is called \emph{strongly} $C^0$-regular there).
\end{proof}

Finally, we will now consider germs of mappings with values in a (global) Banach Lie group.
Therefore, let $H$ be a complex Banach Lie group and let $\h:=L(H)$ be its Lie algebra,
endowed with a norm compatible with the Lie bracket. We denote the adjoint representation
of $H$ on $\h$ by $\smfunc{\Ad_H}{H}{\Aut(\h)\leq \GL(\h).}$ Before we endow the group
$\GermC{K}{H}$ with a manifold structure, we will study the natural action of it
on the locally convex Lie algebra we just defined:
\begin{prop}									\label{prop_adjoint_Germ}
 The pointwise \emph{adjoint action}
\begin{eqnarray*}
\AD\colon  \GermC{K}{H} \times \GermC{K}{\h} &\to& \GermC{K}{\h},\\
\cl{\gamma} . \cl{\eta}
&:= & [ x\mapsto \Ad_H(\gamma(x)).\eta(x)]
\end{eqnarray*}
of the abstract group $\GermC{K}{H}$ on the locally convex Lie algebra $\GermC{K}{\h}$ is
well-defined and for each fixed $[\gamma]\in \GermC{K}{H}$, the linear map
 \[
  \func{\AD(\cl{\gamma})}{\GermC{K}{\h}}{\GermC{K}{\h}}{\cl{\eta}}{\cl{\gamma} . \cl{\eta}}
 \]
 is continuous.
\end{prop}
\begin{proof}
 Let a group element $\cl{\gamma}\in\GermC{K}{H}$ and a Lie algebra element
$\cl{\eta}\in\GermC{K}{\h}$ be given. Then these germs can be represented by complex analytic functions:
 $\smfunc{\gamma}{U_m}{H}$ and $\smfunc{\eta}{U_n}{\h}$, with $m,n\in\N$. We define the action of the
group element $\cl{\gamma}$ on the Lie algebra element $\cl{\eta}$ as the germ $[\gamma].[\eta]$ of the following function:
 \[
  \func{\gamma.\eta}{U_n \cap U_m}{\h}{x}{\Ad_H(\gamma(x) ). \eta(x).}
 \]
 This function is analytic as a composition of analytic maps and it is clear that the germ $\cl{\gamma.\eta}$
does not depend on the chosen functions $\gamma$ and $\eta$ but only on the corresponding germ.

 It remains to show continuity. Let $\cl{\gamma}\in\GermC{K}{H}$ be an element in the abstract group.
It can be represented by an analytic function
 \[
  \smfunc{\gamma}{U}{H}
 \]
 for $U$ an open neighborhood of $K$. Now, consider the composition with the adjoint representation
of the Banach Lie group $H$ which is an analytic map
 \[
  \smfunc{\Ad_H}{H}{\Aut(\h)\subseteq (\BoundOp{\h},\opnorm{\cdot})}.
 \]
 Thus, the composition $\smfunc{\Ad_H\circ \gamma}{U}{\BoundOp{\h}}$ is an analytic function on $U_m$
with values in the Banach space $\BoundOp{\h}$ of bounded operators on $\h$.
 Hence, by the same arguments as in the last section, we may shrink the domain $U$
to a smaller neighborhood $V$ and assume that $\Ad_H\circ \gamma|_V$ is bounded by an $R>0$
on the smaller set. We may assume that $V=U_m$ for a number $m\in\N$.

 Now, for each  $n\ge m$, we consider the linear map
 \[
  \func{A_n}{\BHol{U_n}{\h}}{\BHol{U_n}{\h}}{\eta}{\gamma.\eta}
 \]
with $(\gamma.\eta)(x):=\Ad_H(\gamma(x)).\eta(x)$.
This linear map is continuous, since for each $\eta\in\BHol{U_n}{\h}$, we have
 \[
  \supnorm{ A_n(\eta) } 	 	=	\sup_{x\in U_n} \hnorm{\Ad_H(\gamma(x)). \eta(x)} 
					\leq    \sup_{x\in U_n} \underbrace{\opnorm{\Ad_H(\gamma(x))}}_{\leq R}
                                        \underbrace{\hnorm{\eta(x)}}_{\leq \supnorm{\eta}}
					\leq 	R	\supnorm{\eta}.
 \]
 This shows that each $A_n$ is continuous linear (with $\opnorm{A_n}\leq R$) and by the universal property
of the locally convex direct limit, the direct limit map
 \[
  \func{\AD(\cl{\gamma})}{\GermC{K}{\h}}{\GermC{K}{\h}}{\cl{\eta}}{\cl{\gamma} . \cl{\eta}}
 \]
 is continuous as well.
\end{proof}
Now, we are able to show the main result about the group $\GermC{K}{H}$, in particular that it
carries a natural Lie group structure modelled on the (LB)-space $\GermC{K}{\h}$:
\begin{thm}						\label{thm_GermC_Lie_group_structure}
 On the group $\GermC{K}{H}$, there exists a unique locally convex complex Lie group structure such that the map
 \[
  \func{\EXP}{\GermC{K}{\h}}{\GermC{K}{H}}{\cl{\eta} }{ \cl{\exp_H \circ \eta}}
 \]
 becomes a complex analytic local diffeomorphism.
 The Lie algebra of this Lie group is the complete $($LB$)$-space $\GermC{K}{\h}$
and its exponential map is the map $\smfunc{\EXP}{\GermC{K}{\h}}{\GermC{K}{H}}$ just defined.

Furthermore, the Lie group $\GermC{K}{H}$ is $C^0$-regular.
In particular, it is regular in Milnor's sense.
\end{thm}
To show Theorem \ref{thm_GermC_Lie_group_structure}, we shall use the following fact,
which can be found in \cite[Corollary 1.3.16]{dahmen2011} and is based on a general construction principle
(see e.g.~\cite[Chapter III,\S 1.9, Prop.\ 18]{bourbaki1998}.
\begin{numba}[Construction of a Lie group with a given exponential function]			\label{prop_construction_with_given_exp} \ 
 \begin{itemize}
  \item [(a)] 
 Let $\K\in\set{\R,\C}$ and let $\g$ be a Hausdorff locally convex $\K$-Lie algebra and let
$\Omega_1$ and $\Omega_2$ be open symmetric $0$-neighborhoods in $\g$ such that the BCH-series converges
on $U\times U$ and defines a $C^\omega_\K$-map $\smfunc{\ast}{\Omega_1\times \Omega_1}{\Omega_2}$.
Let $\smfunc{\Phi}{\g}{\cG}$ be an map into an abstract group $\cG$ satisfying
\begin{itemize}
 \item [$\bullet$]  $\Phi|_{\Omega_2}$  is injective.
 \item [$\bullet$]  $\Phi(nx)  = \left(	\Phi(x)	\right)^n	\hbox{ for }n\in\N,x\in\g.$
 \item [$\bullet$]  $\Phi(x*y) = \Phi(x)\cdot\Phi(y) 			\hbox{ for }x,y\in \Omega_1.$
\end{itemize}
 Then there exists a unique $C^\omega_\K$-Lie group structure on
$\cG_0:=\gen{\Phi(\g)}=\gen{\Phi(\Omega_1)}=\gen{\Phi(\Omega_2)}$
such that $\smfunc{\Phi|_{\Omega_1}}{\Omega_1}{\Phi(\Omega_1)}$
becomes a diffeomorphism onto an open subset.

Furthermore, the linear map $\smfunc{T_0\Phi}{\g}{L(\cG_0)}$ is an isomorphism
of locally convex Lie algebras and after identifying $\g$ with $L(\cG_0)$,
we obtain that $\cG_0$ admits a $C^\omega$ exponential function and we have $\exp_G = \Phi$.
  \item [(b)] Let $\cG$ be an abstract group and let $\cN \subseteq \cG$ be a normal subgroup,
carrying a $C^\omega_\K$-Lie group structure. We assume that for each $a\in \cG$ the conjugation map
             \[
              \nnfunc{\cN}{\cN}{g}{a\cdot g\cdot a^{-1}}
             \]
	     is $C^\omega_\K$. Then $\cG$ carries a unique $C^\omega_\K$-structure such that $\cN$ is open in $\cG$.
 \end{itemize}

  \end{numba}

Now, we can prove Theorem \ref{thm_GermC_Lie_group_structure}:
\begin{proof}[Proof of Theorem \ref{thm_GermC_Lie_group_structure}]
 We set $\g:=\GermC{K}{\h}$ and let $\cG:=\GermC{K}{H}$ and let
 \[
  \func{\Phi:=\EXP}{\GermC{K}{\h}}{\GermC{K}{H}}{\cl{\eta} }{ \cl{\exp_H \circ \eta}.}
 \]
 Since $H$ is a Banach Lie group, there is an open ball $\oBallin{\epsilon}{\h}{0}$
such that $\exp_H|_{\oBallin{\epsilon}{\h}{0}}$ is injective.
This implies that the map $\Phi$ is injective on the set
 \[
  \Omega_2 := \setm{\cl{\gamma}\in\GermC{K}{\h}}{\gamma(K)\subseteq \oBallin{\epsilon}{\h}{0}}
 \]
 which is clearly an open neighborhood in $\g=\GermC{K}{\h}$. Furthermore, we know by
Theorem \ref{thm_local_regularity} that there is a symmetric open $0$-neighborhood
in $\Omega_1\subseteq \g$ such that 
the BCH-multiplication on $\g$ converges on $\Omega_1\times\Omega_1$ to a
 $C^\omega_\C$-map $\smfunc{*}{\Omega_1\times \Omega_1}{\g}$. By shrinking $\Omega_1$ if necessary,
we may also assume that $\Omega_1 * \Omega_1 \subseteq \Omega_2$.
 These neighborhoods $\Omega_1$ and $\Omega_2$ now satisfy the hypotheses of part (a) of
Proposition \ref{prop_construction_with_given_exp}.
 Therefore, we obtain an analytic Lie group structure on the group $\cG_0$ generated
by the image of $\Phi$ such that $\Phi$ maps the $0$-neighborhood $\Omega_1\subseteq \g$ diffeomorphically
to the open identity neighborhood $\Phi(\Omega_1)\subseteq \cG_0$.

 Next, we want to use part (b) of Proposition \ref{prop_construction_with_given_exp} to extend the
Lie group structure on $\gen{\Phi(\g)}$ to the whole group $\cG=\GermC{K}{H}$.
 To this end, let $\cl{\gamma}\in \cG$ and let $\cl{\eta}\in \g$ be given. Then we have the formula
 \[
  \cl{\gamma} \cdot \Phi(\cl{\eta}) \cdot (\cl{\gamma})^{-1} = \Phi\left(\AD(\cl{\gamma}).\cl{\eta}\right) \tag{$*$} \label{eqn_r1}
 \]
 which is easily checked. This shows that $\Phi(\g)$ is invariant under the conjugation with elements
of the group $\cG$, hence $\gen{\Phi(\g)}$ is a normal subgroup of $\cG$. It remains to show that the map
 \begin{equation}\label{conjmap}
  \nnfunc{\gen{\Phi(\g)}}{\gen{\Phi(\g)}}{\cl{\xi}}{\cl{\gamma} \cdot \cl{\xi} \cdot (\cl{\gamma})^{-1}}
 \end{equation}
is analytic for each fixed group element $\cl{\gamma}\in\cG$.
Since $\Phi$ is an analytic local diffeomorphism at $0$ (taking $0$ to the identity element)
and $\text{AD}([\gamma])$ is continuous linear (and hence
analytic) by Proposition \ref{prop_adjoint_Germ}, we
deduce from (\ref{eqn_r1}) that the conjugation map (\ref{conjmap})
is analytic on an identity neighborhood
and hence analytic (being a group homomorphism).
Therefore, by part (b) of
Proposition \ref{prop_construction_with_given_exp}, there is a unique Lie group structure
on the group $\GermC{K}{H}$ such that $\gen{\Phi(\GermC{K}{\h})}$ is an open subgroup.

 It remains to show that this Lie group is $C^0$-regular.
 By Lemma 9.5 in \cite{NaS}, we know that a Lie group is $C^0$-regular
if and only if it has an open identity neighborhood which is a $C^0$-regular \emph{local} Lie group.
 Since $\Phi(\Omega_1)$ is an open identity neighborhood of $\cG_0=\gen{\Phi(\g)}$
which itself is open in $\GermC{K}{H}$,
 the regularity of $\GermC{K}{H}$ follows from the regularity of the local group
$(\Omega_1,*)$ which was shown in Theorem  \ref{thm_local_regularity}.
\end{proof}
\noindent{\small {\bf Rafael Dahmen},
Fachbereich Mathematik, Technische Universit\"{a}t Darmstadt,\\
Schlo\ss{}gartenstr.\ 7, 64289 Darmstadt, Germany. Email:\\
\href{mailto:dahmen@mathematik.tu-darmstadt.de}{dahmen@mathematik.tu-darmstadt.de}\\[2mm]
{\bf Helge Gl\"{o}ckner}, Institut f\"{u}r Mathematik,
Universit\"{a}t Paderborn,\\
Warburger Str.\ 100, Germany. Email:
\href{mailto:glockner@math.uni-paderborn.de}{glockner@math.uni-paderborn.de}\\[2mm]
{\bf Alexander Schmeding},
Institutt for matematiske fag, NTNU, 7491 Trondheim,\\
Norway. Email:
\href{mailto:alexander.schmeding@math.ntnu.no}{alexander.schmeding@math.ntnu.no}}\vfill

\begin{thebibliography}{99}
%
%
\bibitem{zaareer2013}
Alzaareer, H., ``Lie Groups of Mappings on Non-Compact Spaces and Manifolds,''
doctoral dissertation, Universit\"{a}t Paderborn, 2013
(see {\tt nbn-resolving.de/urn:nbn:de:hbz:466:2-11572}).
%
%
\bibitem{MR0313810}
Bochnak, J. and J. Siciak,
\emph{Polynomials and multilinear mappings in topological vector spaces},
Studia Math.\ {\bf 39}
(1971), 59--76.
%
%
\bibitem{BS71b}
Bochnak, J. and J. Siciak,
\emph{Analytic functions in topological vector spaces},
Studia Math.\ {\bf 39}
(1971), 77--112.
%
%
\bibitem{bourbaki1998}
Bourbaki, N.,
``Lie Groups and Lie Algebras,'' Chapters 1--3,
Springer, Berlin, 1998.
%
%
\bibitem{BW1959}
Bruhat, F. and H. Whitney,
\emph{Quelques propri\'{e}t\'{e}s fondamentales des ensembles analytiques-r\'{e}els},
Comment.\ Math.\ Helv.\ {\bf 33} (1959), 132--160.
%
%
\bibitem{dahmen2010}
Dahmen, R.,
\emph{Analytic mappings between LB-spaces and applications in infinite-dimensional Lie theory},
Math.\ Z.\ {\bf 266} (2010), no.\ 1, 115--140. 
%
%
\bibitem{dahmen2011}
Dahmen, R.,
``Direct Limit Constructions in Infinite-Dimensional Lie Theory,''
doctoral dissertation, Universit\"{a}t Paderborn, 2011
(see {\tt nbn-resolving.de/urn:nbn:de:hbz:466:2-239}).

%
\bibitem{Dahmen_UnionGroupRegular_JLT}
Dahmen, R.,
\emph{Regularity in {M}ilnor's sense for ascending unions of {B}anach-{L}ie groups}.
\newblock {\em J. Lie Theory}, {\bf 24} (2014), no.\,2, 545--560.
%
%
\bibitem{Engelking1989}
Engelking, R., ``General Topology,''
Sigma Series in Pure Mathematics {\bf 6},
Heldermann, Berlin, ${}^2$1989.
%
%
\bibitem{hg2002}
Gl\"ockner, H.,
\emph{Infinite-dimensional Lie groups without completeness restrictions},
pp.\ 43--59 in: A. Strasburger et al.\ (eds.),
``Geometry and Analysis on Finite- and Infinite-Dimensional Lie Groups,''
Banach Center Publications {\bf 55}, Warsaw, 2002. 
%
%
\bibitem{hg2004}
Gl\"{o}ckner, H.,
\emph{Lie groups of germs of analytic mappings},
pp.\,1--16 in: T. Wurzbacher (ed.),
``Infinite Dimensional Groups and Manifolds,''
IRMA Lect.\ Math.\ Theor.\ Phys.\ {\bf 5}, de Gruyter,
Berlin, 2004. 
%
%
\bibitem{hg2013}
Gl\"{o}ckner, H.,
\emph{Regularity properties of infinite-dimensional Lie groups},
Oberwolfach Rep.\ {\bf 13} (2013), 791--794. 
%
%
\bibitem{hg2013b}
Gl\"{o}ckner, H.,
\emph{Differentiable mappings between spaces of sections}, preprint,
{\tt arXiv:1308.1172}. 
%
%
\bibitem{hg2012}
Gl\"{o}ckner, H.,
%\emph{Notes on regularity properties of infinite-dimensional Lie groups}, preprint,
\emph{Regularity properties of inifinite-dimensional Lie groups
and semiregularity}, manuscript, cf.\ {\tt arXiv:1208.0715}.
%
%
\bibitem{GaN}
Gl\"{o}ckner, H. and K.-H. Neeb,
``Infinite-Dimensional Lie Groups,''
book in preparation.
%
%
\bibitem{grauert}
Grauert, H.,
\emph{On Levi's Problem and the imbedding of real-analytic manifolds},
Annals of Math.\ (2) {\bf 68} (1958), 460--472.
%
%
\bibitem{kaur2013}
Kaur, A., ``Superposition Operators and Associated Infinite-Dimensional Lie Groups,''
Master's thesis, Universit\"{a}t Paderborn, 2013
(advised by H. Gl\"{o}ckner).
%
%
\bibitem{KaM}
Kriegl, A. and P.\,W. Michor,
``The Convenient Setting of Global Analysis,''
Mathematical Surveys and Monographs {\bf 53},
AMS, Providence, 1997.
%
%
\bibitem{Milnor1984}
Milnor, J.,
\emph{Remarks on infinite-dimensional Lie groups}, pp.\ 1007--1057
in: B.\,S. DeWitt and R.\ Stora (eds.),
``Relativity, Groups and Topology II''
(Les Houches, 1983), North-Holland, Amsterdam, 1984. 
%
%
\bibitem{Neeb2006}
Neeb, K.-H.,
\emph{Towards a Lie theory of locally convex groups},
Jpn.\ J. Math.\ {\bf 1} (2006), no.\ 2, 291--468. 
%
%
\bibitem{NaS}
Neeb, K.-H. and H. Salmasian,
\emph{Differentiable vectors and unitary representations of Fr\'{e}chet-Lie supergroups},
Math.\ Z. {\bf 275} (2013), no.\,1--2, 419--451. 
%
%
%\bibitem{NeebSalmasian}
%Neeb, K.-H. and H. Salmasian,
%\emph{Differentiable vectors and unitary representations of
%Fr\'echet--Lie supergroups},
%Math.\ Z. {\bf 275} (2013), no.\,1-2, 419--451.
%
%
\bibitem{NaW}
Neeb, K.-H. and F. Wagemann, 
\emph{Lie group structures on groups of smooth and holomorphic maps on non-compact manifolds},
Geom.\ Dedicata {\bf 134} (2008), 17--60. 
%
%
\bibitem{pieper}
Pieper, T., ``Lie Groups of Unbounded Weighted Mappings,''
Universit\"{a}t Paderborn, Master's thesis, Universit\"at Paderborn, 2014
(advised by H. Gl\"{o}ckner).
%
%
\bibitem{wengenroth}
Wengenroth, J.,
``Derived Functors in Functional Analysis,''
Lecture Notes in Mathematics {\bf 1810},
Springer, Berlin, 2003.
%
%
%
\end{thebibliography}
\end{document}